\documentclass[leqno]{article}
\usepackage[frenchb,english]{babel}
\usepackage[T1]{fontenc}

\usepackage{amsmath}
\usepackage{amssymb}
\usepackage{amsfonts}
\usepackage{enumerate}
\usepackage{vmargin}
\usepackage[all]{xy}
\usepackage{mathrsfs}
\usepackage{mathtools}
\usepackage{lmodern}
\usepackage[pagebackref=true]{hyperref}

\setmarginsrb{3cm}{3cm}{3.5cm}{3cm}{0cm}{0cm}{1.5cm}{3cm}

\newcommand{\B}{\ensuremath{\mathrm{B}}}
\newcommand{\C}{\ensuremath{\mathbb{C}}}
\newcommand{\E}{\ensuremath{\mathbb{E}}}
\newcommand{\F}{\ensuremath{\mathbb{F}}}
\let\H\relax 
\newcommand{\H}{\mathrm{H}} 

\let\L\relax 
\newcommand{\L}{\mathrm{L}}

\newcommand{\Q}{\ensuremath{\mathbb{Q}}}
\newcommand{\R}{\ensuremath{\mathbb{R}}}
\let\S\relax 
\newcommand{\S}{\mathrm{S}}
\let\P\relax 
\newcommand{\P}{\mathrm{P}}
\newcommand{\p}{\mathrm{p}}
\newcommand{\T}{\ensuremath{\mathbb{T}}}

\newcommand{\Z}{\ensuremath{\mathbb{Z}}}
\newcommand{\e}{\mathrm{e}}

\newcommand{\w}{\mathrm{w}} 
\renewcommand{\d}{\mathop{}\mathopen{}\mathrm{d}} 

\newcommand{\ran}{\ensuremath{\mathop{\rm Ran\,}}}
\newcommand{\tr}{\ensuremath{\mathop{\rm Tr\,}\nolimits}}

\let\i\relax 
\newcommand{\i}{\mathrm{i}}
\newcommand{\Id}{\mathrm{Id}} 
\newcommand{\VN}{\mathrm{VN}} 

\renewcommand{\leq}{\ensuremath{\leqslant}}
\renewcommand{\geq}{\ensuremath{\geqslant}}
\newcommand{\qed}{\hfill \vrule height6pt  width6pt depth0pt}
\newcommand{\norm}[1]{ \| #1  \|}
\newcommand{\bnorm}[1]{ \big\| #1  \big\|}

\newcommand{\Bgnorm}[1]{ \Bigg\| #1  \Bigg\|}
\newcommand{\xra}{\xrightarrow} 
\newcommand{\co}{\colon}

\newcommand{\ot}{\otimes}
\newcommand{\ovl}{\overline}
\newcommand{\otvn}{\ovl\ot}
\newcommand{\ul}{\mathcal{U}}

\newcommand{\epsi}{\varepsilon}
\newcommand{\QWEP}{\mathrm{QWEP}}
\newcommand{\OUMD}{{\rm OUMD}}

\DeclareMathOperator{\Ran}{Ran} 
\DeclareMathOperator{\Conv}{Conv} 
\DeclareMathOperator{\ad}{ad}
\DeclareMathOperator{\Dom}{Dom}
\newtheorem{thm}{Theorem}[section]
\newtheorem{defi}[thm]{Definition}
\newtheorem{prop}[thm]{Proposition}

\newtheorem{cor}[thm]{Corollary}
\newtheorem{lemma}[thm]{Lemma}

\newtheorem{remark}[thm]{Remark}
\newenvironment{rk}{\begin{remark}\rm}{\end{remark}}

\newenvironment{proof}[1][]{\noindent {\it Proof #1}\hspace{-0.04cm}: }{\hbox{~}\qed
\smallskip
}

\numberwithin{equation}{section}

\begin{document}
\selectlanguage{english}
\title{\bfseries{Dilations of semigroups on von Neumann algebras and noncommutative $\L^p$-spaces}}
\date{}
\author{\bfseries{Cédric Arhancet}}

\maketitle

\begin{abstract}
We prove that any weak* continuous semigroup $(T_t)_{t \geq 0}$ of factorizable Markov maps acting on a von Neumann algebra $M$ equipped with a normal faithful state can be dilated by a group of Markov $*$-automorphisms analogous to the case of a single factorizable Markov operator, which is an optimal result. We also give a version of this result for strongly continuous semigroups of operators acting on noncommutative $\L^p$-spaces and examples of semigroups to which the results of this paper can be applied. Our results imply the boundedness of the McIntosh's $\H^\infty$ functional calculus of the generators of these semigroups on the associated noncommutative $\L^p$-spaces generalising some previous work from Junge, Le Merdy and Xu. Finally, we also give concrete  dilations for Poisson semigroups which are even new in the case of $\R^n$.

\bigskip
\end{abstract}


\makeatletter
 \renewcommand{\@makefntext}[1]{#1}
 \makeatother
 \footnotetext{
 2010 {\it Mathematics subject classification:}
 Primary 47A20, 47D03, 46L51 ; Secondary 47D07.
\\
{\it Key words and phrases}: Markov semigroups, dilations, von Neumann algebras, functional calculus.}

\section{Introduction}
\label{sec:Introduction}

The study of dilations of operators is of central importance in operator theory and has a long tradition in functional analysis. Indeed, dilations are very useful tools which allow to reduce general studies of operators to more tractable ones. Suppose $1<p<\infty$. In the spirit of Sz.-Nagy's dilation theorem for contractions on Hilbert spaces, a classical result from the 70's essentially due to Akcoglu \cite{AkS} (see also \cite{Pel1} and \cite{KNP}) says that a positive contraction $T \co \L^p(\Omega) \to \L^p(\Omega)$ on an $\L^p$-space $\L^p(\Omega)$ admits a positive isometric dilation $U$ on a bigger $\L^p$-space than the initial $\L^p$-space, i.e. there exist another measure space $\Omega'$, two positive contractions $J \co \L^p(\Omega) \to \L^p(\Omega')$ and $P \co \L^p(\Omega') \to \L^p(\Omega)$ and a positive invertible isometry $U \co \L^p(\Omega') \to \L^p(\Omega')$ such that 
$$
T^k = PU^kJ
$$ 
for any integer $k \geq 0$. Note that in this situation, $J$ is an isometric embedding whereas $JP$ is a contractive projection.

Later, Fendler \cite{Fen1} proved a continuous version of this result for any strongly continuous semigroup $(T_t)_{t \geq 0}$ of positive contractions on an $\L^p$-space $\L^p(\Omega)$. More precisely, this theorem says that there exists a measure space $\Omega'$, two positive contractions $J \co \L^p(\Omega) \to \L^p(\Omega')$ and $P \co \L^p(\Omega') \to \L^p(\Omega)$ and a strongly continuous group of positive invertible isometries $(U_t)_{t \in \R}$ on $\L^p(\Omega')$ such that 
$$
T_t = PU_tJ
$$ 
for any $t \geq 0$, see also \cite{Fen2} and \cite{Kon1}. 

In the noncommutative setting, measure spaces and $\L^p$-spaces are replaced by von Neumann algebras and noncommutative $\L^p$-spaces and positive maps by completely positive maps. In their remarkable paper \cite{JLM}, Junge and Le Merdy showed that there exists no ``reasonable'' analog of Akcoglu result for completely positive contractions acting on noncommutative $\L^p$-spaces. It is a \textit{striking} difference with the world of classical (=commutative) $\L^p$-spaces of measure spaces. 

Independently, Kümmerer, Maasen, Haagerup and Musat introduced and studied dilations of well-behaved completely positive unital operators on noncommutative probability spaces (=von Neumann algebras equipped with states), the so-called Markov operators \cite{Kum1} \cite{Kum2} \cite{Kum3} \cite{KuM} \cite{HaM} \cite{HaM2}. These dilations induce dilations on the associated noncommutative $\L^p$-spaces and we can see these dilations as reasonable substitutes of dilations furnished by Akcoglu's theorem. The following definition of these operators is considered in \cite{AnD}, \cite{HaM} and \cite{Ric}.

\begin{defi}
\label{def-Markov}
Let $(M,\phi)$ and $(N,\psi)$ be von Neumann algebras equipped with normal faithful states $\phi$ and $\psi$, respectively. A linear map $T\co M \to N$ is called a $(\phi,\psi)$-Markov map if
\begin{enumerate}
\item [$(1)$] $T$ is completely positive
\item [$(2)$] $T$ is unital
\item [$(3)$] $\psi\circ T=\phi$
\item [$(4)$] for any $t \in \R$ we have $T \circ \sigma_t^{\phi}=\sigma_t^\psi \circ T$ where $(\sigma_t^{\phi})_{t \in \R}$ and $(\sigma_t^{\psi})_{t \in \R}$ denote the automorphism groups of the states $\phi$ and $\psi$, respectively.
\end{enumerate}
\end{defi}
In particular, when $(M,\phi)=(N,\psi)$, we say that $T$ is a $\phi$-Markov map. Such an operator $T$ induces a contraction $T_p \co \L^p(M) \to \L^p(N)$ between the associated noncommutative $\L^p$-spaces $\L^p(M)$ and $\L^p(N)$ for any $1 \leq p<\infty$, see for example \cite[Theorem 5.1]{HJX}. 

The following definition is essentially due to Kümmerer, see \cite[Definition 2.1.1]{Kum2} and \cite[Definition 4.1]{HaM} and describes a form of dilation suitable for Markov maps in the style of Akcoglu's dilations. We refer to \cite{Kum1} \cite{Kum2} \cite{Kum3} \cite{KuM} for physical interpretations of this notion.

\begin{defi}
\label{def-T-dilatable}
Let $M$ be a von Neumann algebra equipped with a normal faithful finite state $\phi$ and let $T \co M \to  M$ be a $\phi$-Markov map. We say that $T$ is dilatable if there exists a von Neumann algebra $N$ equipped with a normal faithful state $\psi$, a $*$-automorphism $U$ of $N$ leaving $\psi$ invariant and a $(\phi,\psi)$-Markov $*$-monomorphism $J \co  M \to N$ satisfying
$$
T^k 
=\E U^k J, \qquad k \geq 0.
$$
where $\E=J^* \co N \to M$ is the canonical faithful normal conditional expectation preserving the states associated with $J$.
\end{defi}
In the situation of Definition \ref{def-T-dilatable}, the $*$-automorphism $U \co N \to N$ commutes with the modular automorphism group $(\sigma_t^\psi)_{t \in \R}$ of $\psi$ by \cite[Corollary 1.4 page 95]{Tak2}. Consequently, $U$ is a $\psi$-Markov map.

Note that Haagerup and Musat \cite[Theorem 4.4]{HaM} have succeeded in characterizing dilatable Markov maps. Indeed, they proved that a $\phi$-Markov map $T \co M \to M$ is dilatable if and only if $T$ is factorizable in the sense of \cite{AnD}, i.e. there exists a von Neumann algebra $N$ equipped with a faithful normal state $\psi$ and $(\phi,\psi)$-Markov $*$-monomorphisms $J_0 \co M \to N$ and $J_1 \co M \to N$ satisfying $T=J_0^*\circ J_1$ where $J_0^*$ denotes the adjoint defined below in \eqref{Adjoint}.

Now, we introduce the continuous version of this definition from \cite[Definition 1.3]{Arh2} inspired by Fendler result, see also \cite[Definition page 4]{KuM}. 
\begin{defi}
\label{Def QWEP dilatable} 
Let $M$ be a von Neumann algebra equipped with a normal faithful state $\phi$. Let $(T_t)_{t\geq 0}$ be a weak* continuous semigroup of $\phi$-Markov maps on $M$. We say that the semigroup is dilatable if there exists a von Neumann algebra $N$ equipped with a normal faithful state $\psi$, a weak* continuous group $(U_t)_{t \in \R}$ of $*$-automorphisms of $N$ leaving $\psi$ invariant, a $(\phi,\psi)$-Markov $*$-monomorphism $J \co M \to N$  satisfying
\begin{equation}
\label{Equa-Dilation-semigroup}
T_t
=\E U_t J,\qquad t \geq 0,
\end{equation}
where $\E=J^* \co N \to M$ is the canonical faithful normal conditional expectation preserving the states associated with $J$.
\end{defi}

Note that such a dilation induces an isometric dilation similar to the one of Fendler's theorem for the strongly continuous semigroup induced by the semigroup $(T_t)_{t\geq 0}$ on the associated noncommutative $\L^p$-space $\L^p(M)$ for any $1 \leq p <\infty$.

Haagerup and Musat have given a necessary condition for the existence of such a dilation. Indeed, by \cite[Theorem 4.4]{HaM}, the situation of the above definition implies that each operator $T_t \co M \to M$ is factorizable. Moreover, they constructed \cite[Theorem 3.4]{HaM} a semigroup $(T_t)_{t \geq 0}$ of Markov maps such that not all $T_t$ are factorizable, hence without dilation, leaving implicitly open the question of whether this condition of factorizability is sufficient for the existence of a dilation. Our main result is the following theorem which shows that this condition is equally sufficient.


\begin{thm}
\label{thm-dilation-semigroup-weak-star}
Let $M$ be a von Neumann algebra equipped with a normal faithful state $\phi$. Let $(T_t)_{t \geq 0}$ be a weak* continuous semigroup of factorizable $\phi$-Markov maps on $M$. Then the semigroup $(T_t)_{t \geq 0}$ is dilatable.
\end{thm}

The semigroups of selfadjoint Markov Fourier multipliers plays a fundamental role in noncommutative harmonic analysis and operator algebras (see e.g. \cite{Haa1}) and this result implies, by combining some results of Ricard \cite{Ric}, in particular that these semigroups are dilatable, see Corollary \ref{thm-Dilation-semigroup-Fourier-multipliers}. We also prove Theorem \ref{thm-dilation-Fendler} which is a variant of this result for noncommutative $\L^p$-spaces useful even for non-$\sigma$-finite von Neumann algebras. See also \cite{JRS} for related results.


One of the important consequences of Fendler's theorem is the boundedness, for the generators of strongly continuous semigroup $(T_t)_{t \geq 0}$ of positive contractions, of a bounded $\H^\infty$ functional calculus which is a fundamental tool in various areas: harmonic analysis of semigroups, multiplier theory, Kato's square root problem, maximal regularity in parabolic equations, control theory, etc. For detailed information, we refer the reader to \cite{Haa}, \cite{JMX}, \cite{KW} and to the recent book \cite{HvNVW2} and to the references therein. Our theorem also gives a similar result on $\H^\infty$ functional calculus in the noncommutative context. In particular, our natural and straightforward approach allows us to recover some of the main results of the fundamental memoir \cite{JMX} on functional calculus in few pages. More significantly, combined with some results of the author \cite{Arh2}, it even allows to give generalizations to vector-valued noncommutative $\L^p$ spaces. We refer to Section \ref{sec:Applications} for precise statements of our results.

In the opposite direction of our abstract and non-constructive proof of our main result, we also present some very concrete dilations for some particular cases: classical Poisson semigroups on $\L^\infty(\R^n)$ and on $\L^\infty(\T^n)$ and noncommutative Poisson semigroups on von Neumann algebras $\VN(\F_n)$ of free groups $\F_n$ and semigroups of Hamiltonians.

The paper is organized as follows. The next section gives background. Section \ref{sec:Dilations-VN} gives a proof of Theorem \ref{thm-dilation-semigroup-weak-star}. In Section \ref{Concrete}, we present some concrete dilations for some classical semigroups. In the following section \ref{sec:dilations-Lp}, we describe and prove a noncommutative $\L^p$ analog of this result. In section \ref{sec:exemples}, we consider the case of semigroups of Fourier multipliers. Finally, we conclude in section \ref{sec:Applications} with applications of our results to functional calculus.

\section{Preliminaries}
\label{sec:Preliminaries}
\paragraph{Noncommutative $\L^p$-spaces}
We use Haagerup noncommutative $\L^p$-spaces. We refer to the survey \cite{PiX} and to the papers \cite{Ray2}, \cite{JMX} and \cite{Pis5} for more information. 

Let $M$ be a von Neumann algebra equipped with a normal faithful state $\varphi$. Let $D_{\varphi}$ the density operator associated with $\varphi$. If $1 \leq p <\infty$, recall that by \cite[Lemma 1.1]{JuX1} and \cite[Corollary 4]{Wat1}, $D^{\frac{1}{2p}}_{\varphi} M  D^{\frac{1}{2p}}_{\varphi}$ is a dense subspace of $\L^p(M)$.

Suppose that $N$ is another von Neumann algebra equipped with a normal faithful state $\psi$. Consider a unital positive map $T \co M \to N$ such that $\psi(T(x)) =\varphi(x)$ for any $x \in M_+$. Given $1 \leq p<\infty$ define
\begin{equation}
\label{Map-extension-Lp}
\begin{array}{cccc}
  T_p  \co &    D^{\frac{1}{2p}}_{\varphi} M  D^{\frac{1}{2p}}_{\varphi}  &  \longrightarrow   &  D^{\frac{1}{2p}}_{\psi} N  D^{\frac{1}{2p}}_{\psi}  \\
           &   D^{\frac{1}{2p}}_{\varphi} x D^{\frac{1}{2p}}_{\varphi}  & \longmapsto &  D^{\frac{1}{2p}}_{\psi}T(x) D^{\frac{1}{2p}}_{\psi}  \\
\end{array}.
\end{equation}
By \cite[Theorem 5.1]{HJX}, the map $T_{p}$ above extends to a contractive map from $\L^p(M)$ into $\L^p(N)$.

\paragraph{Markov operators}
By \cite[Proposition 5.4]{HJX} and \cite[Remark page 249]{AcC}, a linear map $T \co M \to N$ satisfying the first three conditions of Definition \ref{def-Markov} is automatically normal. If, moreover, condition $(4)$ is satisfied, then it was proved in \cite{AcC} (see also \cite[Lemma 2.5]{AnD}) that there exists a unique linear map $T^* \co N \to M$ such that
\begin{equation}
\label{Adjoint}
\phi\big(T^*(y)x\big) 
=\psi\big(y T(x)\big), \quad x\in M, y\in N. 
\end{equation}
It is well-known that $T^*$ is a $(\psi,\phi)$-Markov map. A special case of interest is the one of a $(\phi,\psi)$-Markov map $J \co M \to N$ which is a $*$-monomorphism. In this case, the adjoint $J^* \co N \to M$ is the canonical normal faithful conditional expectation $\E \co N \to M$ preserving the states associated with $J$, see \cite[Remark 1.2]{HaM}. Moreover, we say that a $\phi$-Markov map $T \co M \to M$ is selfadjoint if $T=T^*$.

It is not difficult to prove the following elementary observation.  

\begin{lemma} 
\label{Lemma-injectivity}
Let $(M,\phi)$ and $(N,\psi)$ be von Neumann algebras equipped with normal faithful states $\phi$ and $\psi$, respectively. A $(\phi,\psi)$-Markov $*$-homomorphism $T \co M\to N$ is always injective. 
\end{lemma}

\begin{proof}
Consider $x \in M^+$. Suppose $T(x)=0$. We have $\phi(x)=\psi(T(x))=0$. Hence $x=0$ by the positivity of $T$ and the faithfulness of $\phi$. Now if $y \in M$ satisfies $T(y)=0$. We have $T(y)^*T(y)=0$. Since $T$ is a $*$-homomorphism, we infer that $T(y^*y)=0$. We deduce that $y^*y=0$ and therefore that $y=0$.
\end{proof}

We will use the following Lemma left to the reader.

\begin{lemma}
\label{Lemma-prop-stabilite}
Let $(M,\phi)$, $(N,\psi)$ and $(P,\varphi)$ be von Neumann algebras equipped with normal faithful states $\phi$, $\psi$ and $\varphi$, respectively. The set of $(\phi,\psi)$-Markov maps from $M$ into $N$ is convex, closed for the point weak* topology\footnote{\thefootnote. If $X$ is a dual Banach space with predual $X_*$, the point weak* topology on $\B(X)$ is the topology of pointwise convergence on $X$ endowed with the $\sigma(X,X_*)$-topology, i.e. a net $(T_i)$ in $\B(X)$ converges to a point $T \in \B(X)$ for this topology if and only if for any $x \in X$ and any $y \in X_*$ we have $\langle y, T_i(x) \rangle_{X_*,X} \xra[i]{}\langle y, T(x) \rangle_{X_*,X}$.}. Moreover the composition of a $(\phi,\psi)$-Markov map and of a $(\psi,\varphi)$-Markov map is a $(\phi,\varphi)$-Markov map.
\end{lemma}

\paragraph{Ultraproducts of Banach spaces}
Let $(X_n)_{n \geq 1}$ be a sequence of Banach spaces, and let $\ell^{\infty}(\mathbb{N},X_n)$ be the Banach space of all sequences $(x_n)_{n \geq 1} \in \prod_{n=1}^{\infty} X_n$ with $\sup_{n \geq 1} \norm{x_n}_{X_n} < \infty$ equipped with the norm $\norm{(x_n)_{n \geq 1}}_{\ell^{\infty}(\mathbb{N},X_n)}=\sup_{n \geq 1} \norm{x_n}_{X_n}$. Let $\ul$ be a free ultrafilter on $\mathbb{N}$. The Banach space ultraproduct $(X_n)^{\ul}$ is defined as the quotient $\ell^{\infty}(\mathbb{N},X_n)/\mathcal{J}_{\ul}$, where $\mathcal{J}_{\ul}$ is the closed subspace of all $(x_n)_{n \geq 1}\in \ell^{\infty}(\mathbb{N},X_n)$ which satisfies $\lim_{n \to \ul} \norm{x_n}_{X_n}=0$. An element of $(X_n)^{\ul}$ represented by $(x_n)_{n \geq 1} \in \ell^{\infty}(\mathbb{N},E)$ is written as $(x_n)^{\ul}$. For any $(x_n)^{\ul} \in (X_n)^{\ul}$, one has $\norm{(x_n)^{\ul}}=\lim_{n \to \ul} \norm{x_n}_{X_n}$.  If $(T_n \co X_n  \to Y_n)_{n \geq 1}$ is a bounded sequence of bounded linear operators, we can define the ultraproduct map $T \co (X_n)^\ul \to (Y_n)^\ul $, $(x_n)^\ul \mapsto (T_n(x_n))^\ul$. We refer to \cite[section 8]{DJT} for more information.

If $1 \leq p <\infty$, a ultraproduct of noncommutative $\L^p$-spaces is a noncommutative $\L^p$-space, see \cite{Ray1}. However, the Banach space ultraproduct of von Neumann algebras is not a von Neumann algebra in general. 

\paragraph{Ultraproducts of von Neumann algebras}

If $\phi$ is a normal faithful state on a von Neumann algebra $M$, we define $\norm{\cdot}_\phi^\sharp$ and $\norm{\cdot}_\phi$ by
\begin{equation}
\label{Norms}
\norm{x}_\phi^\sharp
=\big(\phi(x^*x+xx^*)\big)^{\frac{1}{2}},\quad \norm{x}_\phi
=\big(\phi(x^*x)\big)^{\frac{1}{2}},\quad x \in M.
\end{equation}
On bounded sets, the (relative) topology associated to the norm $\norm{\cdot}_\phi^\sharp$ coincide with  the (relative) topology of $\sigma$-strong* topology, see \cite[Section 2]{AnH}, \cite[Lemma 1.11.2]{Bin} and the discussion from \cite[page 20]{Bin}. 

Let us now define the Ocneanu ultraproduct $(M_{n},\phi_n)^\ul$ of a sequence $(M_n,\phi_n)_{n \geq 1}$ of $\sigma$-finite von Neumann algebras equipped with normal faithful states $\phi_n$ with respect to a free ultrafilter $\mathcal{U}$ over $\mathbb{N}$. Let $\ell^{\infty}(\mathbb{N},M_n)$ be the C*-algebra of sequences $(x_{n})_{n \geq 1}$ of $\prod\nolimits_{n=1}^{\infty} M_{n}$ such that $\sup_{n \geq 1}\left\Vert x_{n}\right\Vert_{M_n} <+\infty$ endowed with the norm $\left\Vert (x_{n})\right\Vert_{\ell^{\infty}(\mathbb{N},M_n)} =\sup_{n  \geq 1}\left\Vert x_{n}\right\Vert_{M_n}$. Let $\ul$ be free ultrafilter on $\mathbb{N}$. We let
$$
\mathcal{I}_{\ul}(M_n,\phi_n)
:=\left\{ (x_n)_{n \geq 1} \in \ell^{\infty}(\mathbb{N},M_n)\ : \ \norm{x_n}_{\phi_n}^{\sharp}\xra[n \to \ul]{} 0 \right\}.
$$
If $(x_n)_{n \geq 1}$ is a bounded sequence of $\ell^{\infty}(\mathbb{N},M_n)$, note the following equivalence
\begin{equation}
\label{Equivalence}            
\norm{x_n}_{\phi_n}^{\sharp} \xra[n \to \ul]{} 0 
\iff \norm{x_n}_{\phi_n} \xra[n \to \ul]{} 0
\text{ and } \norm{x_n^*}_{\phi_n} \xra[n \to \ul]{} 0.
\end{equation}
With the abbreviated notation $\mathcal{I}_{\ul}$ for $\mathcal{I}_{\ul}(M_n,\phi_n)$, we define the following subspace
$$
\mathcal{M}^{\ul}(M_n,\phi_n)
:=\big\{(x_n)_{n \geq 1} \in \ell^{\infty}(\mathbb{N},M_n)\ :\ (x_n) \mathcal{I}_{\ul} \subset \mathcal{I}_{\ul}, \text{ and }\mathcal{I}_{\ul}(x_n) \subset \mathcal{I}_{\ul}\big\}.
$$
Then $\mathcal{M}^{\ul}(M_n,\phi_n)$ is a C*-algebra (with pointwise operations and supremum norm) in which $\mathcal{I}_{\ul}(M_n,\phi_n)$ is a closed ideal. We then define the quotient C*-algebra 
$$
(M_n,\phi_n)^{\ul}
:=\mathcal{M}^{\ul}(M_n,\phi_n)/\mathcal{I}_{\ul}(M_n,\phi_n).
$$
Then $(M_n,\phi_n)^{\ul}$ is a W*-algebra. We denote the image of an element $(x_n)_{n \geq 1}$ of $\mathcal{M}^{\ul}(M_n,\phi_n)$ in $(M_n,\phi_n)^{\ul}$ by $(x_n)^{\ul}$. Finally, we can define a normal faithful state $(\phi_n)^{\ul}$ on $(M_n,\phi_n)^{\ul}$ by
$$
(\phi_n)^{\ul}\big((x_n)^{\ul}\big)
:=\lim_{n \to \ul} \phi_n(x_n),\quad (x_n)^{\ul} \in (M_n,\phi_n)^{\ul}.
$$
See \cite{AHW}, \cite{AnH} and \cite {Ocn1} for more information. In the particular case where the von Neumann algebras $M_n$ are finite and where the states $\phi_n$ are normal faithful tracial states, we recover the tracial ultraproduct described in \cite[Section 9.10]{Pis6} (see also \cite[Appendix A]{SS}) and $(\phi_n)^{\ul}$ is a normal faithful tracial state. Finally, in the case of constant sequence $M_n = M$, $\phi_n =\phi$, the ultraproduct $(M_n,\phi_n)^{\ul}$ is written as $M^\ul$ and called the ultrapower of $(M,\phi)$ and we write $\phi^{\ul}$ for $(\phi)^\ul$.

The modular automorphism group $(\sigma_t^{(\phi_n)^{\ul}})_{t \in \R}$ of the ultraproduct state $(\phi_n)^{\ul}$ is described in \cite[Theorem 4.1]{AnH}. For any $t \in \R$ and any $(x_n)^{\ul} \in M^{\ul}$, we have 
\begin{equation}
\label{Modular-ultraproduct}
\sigma_t^{(\phi_n)^{\ul}}\big((x_n)^{\ul}\big)
=\big(\sigma_t^{\phi_n}(x_n)\big)^{\ul}.	
\end{equation}


\begin{prop}
\label{Prop-ultramap}
Let $(M_n,\phi_n)_{n \geq 1}$ and $(N_n,\psi_n)_{n \geq 1}$ be sequences of $\sigma$-finite von Neumann algebras equipped with normal faithful states. If $(J_n \co M_n \to N_n)_{n \geq 1}$ is a sequence of $(\phi_n,\psi_n)$-Markov $*$-monomorphism then the map $(J_n)^\ul \co (M_{n})^\ul \to (N_{n})^\ul$, $(x_n)^\ul \mapsto (J_n(x_n))^\ul$ is a well-defined $((\phi_n)^\ul,(\psi_n)^\ul)$-Markov $*$-monomorphism. Moreover, if we denote by $\E_n \co N_n \to M_n$ the canonical normal faithful conditional expectation preserving the states associated with $J_n$ then the canonical conditional expectation $(\E_n)^\ul \stackrel{\text{def}}{=} ((J_n)^\ul)^* \co (N_{n})^\ul \to (M_{n})^\ul$ satisfies for any element $(y_n)^\ul$ of $(N_{n})^\ul$
\begin{equation}
\label{Esperance-ul}
(\E_n)^\ul\big((y_n)^\ul\big)
=\big(\E_n(y_n)\big)^\ul.	
\end{equation}
\end{prop}

\begin{proof}
We begin to prove that if $(x_n)$ is an element of $\mathcal{M}^{\ul}(M_n,\phi_n)$ then the sequence $(J(x_n))$ is an element of $\mathcal{M}^{\ul}(N_n,\psi_n)$. If $(y_n)$ is an element of $\mathcal{I}_{\ul}(N_n,\psi_n)$, using \eqref{Equivalence} it suffices to check that $(J_n(x_n) y_n) \in \mathcal{I}_{\ul}(N_n,\psi_n)$ and that $(y_n J_n(x_n)) \in \mathcal{I}_{\ul}(N_n,\psi_n)$, that is
$$
\lim_{n \to \ul} \bnorm{(J_n(x_n)y_n)^*}_{\psi_n} 
=0, \quad
\lim_{n \to \ul} \bnorm{J_n(x_n) y_n}_{\psi_n} 
=0, \quad
\lim_{n \to \ul} \bnorm{(y_n J_n(x_n))^*}_{\psi_n} 
=0
$$
and
$$
\lim_{n \to \ul} \bnorm{y_n J_n(x_n)}_{\psi_n} 
=0.
$$
We start with the first limit. We begin to note that
$$
\bnorm{(\E_n(y_n y_n^*))^{\frac{1}{2}}}_{\phi_n} 
=\big(\phi_n\big(\E_n(y_n y_n^*)\big)\big)^{\frac{1}{2}}
=\big(\psi_n(y_n y_n^*)\big)^{\frac{1}{2}}
\xra[n \to \ul]{} 0.
$$
Hence the sequence $((\E_n(y_n y_n^*))^{\frac{1}{2}})$ belongs to $\mathcal{I}_{\ul}(M_n,\phi_n)$. Since $(x_n)$ is an element of $\mathcal{M}^{\ul}(M_n,\phi_n)$, we deduce that the sequence $(x_n \E_n(y_ny_n^*)^{\frac{1}{2}})$ belongs to $\mathcal{I}_{\ul}(M_n,\phi_n)$. Consequently, we obtain
\begin{align*}
\MoveEqLeft
 \bnorm{(J_n(x_n) y_n)^*}_{\psi_n}          
		=\psi_n\big( J_n(x_n) y_n (J_n(x_n) y_n)^*\big)^{\frac{1}{2}}
		=\psi_n\big( J_n(x_n) y_ny_n^* J_n(x_n^*)\big)^{\frac{1}{2}}\\
		&=\phi_n\big(\E_n( J_n(x_n) y_ny_n^* J_n(x_n^*))\big)^{\frac{1}{2}}
		=\phi_n\big(x_n \E_n(y_ny_n^*) x_n^*\big)^{\frac{1}{2}}
		=\phi_n\big(x_n \E_n(y_ny_n^*)^{\frac{1}{2}}\E_n(y_ny_n^*)^{\frac{1}{2}} x_n^*\big)^{\frac{1}{2}}\\
		&=\bnorm{\big(x_n \E_n(y_ny_n^*)^{\frac{1}{2}}\big)^*}_{\phi_n}
		\xra[n \to \ul]{} 0.
\end{align*}
For the second limit, we first observe that
$$
\norm{J_n(x_n) y_n}_{\psi_n}          
=\psi_n\big((J_n(x_n) y_n)^*J_n(x_n) y_n\big)^{\frac{1}{2}}
=\psi_n\big(y_n^* J_n(x_n)^* J_n(x_n) y_n\big)^{\frac{1}{2}}.
$$
Since the sequence $(J_n(x_n))$ is bounded, there exists a constant $C$ such that $\norm{J_n(x_n)}_N  \leq C$ for any integer $n$. Using \cite[3.2 (i) page 269]{ScW1}, we obtain $y_n^* J_n(x_n)^* J_n(x_n) y_n \leq C y_n^* y_n$ and finally
\begin{equation}
	\label{Similar-reasoning}
\psi_n\big(y_n^* J_n(x_n)^* J_n(x_n) y_n\big)^{\frac{1}{2}} 
\leq \sqrt{C}\psi_n\big(y_n^* y_n\big)^{\frac{1}{2}}
\xra[n \to \ul]{} 0.	
\end{equation}
We conclude that $\norm{J_n(x_n) y_n}_{\psi_n} 
\xra[n \to \ul]{} 0$. The two last limits are similar and are left to the reader.

Now the $J_n$'s induce a $*$-homomorphism $\oplus J_n \co \ell^{\infty}(\mathbb{N},M_n) \to \ell^{\infty}(\mathbb{N},N_n)$, $(x_1,x_2,\ldots) \mapsto (J_1(x_1),J_2(x_2),\ldots)$. Using the beginning of the proof, the restriction of this map $\oplus J_n$ gives a well-defined map $\mathcal{M}^{\ul}(M_n,\phi_n) \to \mathcal{M}^{\ul}(N_n,\psi_n)$ an therefore a map $J \co \mathcal{M}^{\ul}(M_n,\phi_n) \to (N_{n})^\ul$. Now if $(x_n)$ is an element of $\ell^{\infty}(\mathbb{N},M_n)$ we have
$$
\norm{J_n(x_n)}_{\psi_n}          
=\psi_n\big((J_n(x_n))^*J_n(x_n)\big)^{\frac{1}{2}}
=\psi_n\big((J_n(x_n^*x_n)\big)^{\frac{1}{2}}
=\phi_n\big(x_n^*x_n\big)^{\frac{1}{2}}
$$
and
$$
\bnorm{(J_n(x_n))^*}_{\psi_n}          
=\psi_n\big(J_n(x_n)(J_n(x_n))^*\big)^{\frac{1}{2}}
=\psi_n\big((J_n(x_nx_n^*)\big)^{\frac{1}{2}}
=\phi_n\big(x_nx_n^*\big)^{\frac{1}{2}}.
$$
We deduce that the kernel $\ker J$ of $J$ is equal to $\mathcal{I}_{\ul}(N_n,\psi_n)$. Hence we obtain a well-defined quotient map $(J_n)^\ul \co (M_{n})^\ul \to (N_{n})^\ul$, $(x_n)^\ul \mapsto (J_n(x_n))^\ul$ which is clearly a unital $*$-monomorphism. Now, we will show that $(J_n)^\ul \co (M_{n})^\ul \to (N_{n})^\ul$ is a $((\phi_n)^\ul,(\psi_n)^\ul)$-Markov map. Using \eqref{Modular-ultraproduct}, for any $t \in \R$ and any element $(x_n)^\ul$ of $(M_{n})^\ul$, we have
\begin{align*}
\MoveEqLeft
 (J_n)^\ul \circ  \sigma_t^{(\phi_n)^{\ul}} \big((x_n)^\ul\big)         
		=(J_n)^\ul\big(\sigma_t^{\phi_n}(x_n)\big)^{\ul}
		=\big(J_n (\sigma_t^{\phi_n}(x_n))\big)^{\ul}\\
		&=\big(\sigma_t^{\psi_n}(J_n(x_n))\big)^{\ul}
		= \sigma_t^{(\psi_n)^{\ul}} \circ (J_n)^\ul\big((x_n)^\ul\big).
\end{align*}
Hence $(J_n)^\ul \circ \sigma_t^{(\phi_n)^{\ul}}=\sigma_t^{(\psi_n)^{\ul}} \circ (J_n)^\ul$, i.e. the map $(J_n)^\ul$ commutes with the modular automorphism groups. Moreover
\begin{align*}
\MoveEqLeft
  (\psi_n)^{\ul}\left((J_n)^\ul\big((x_n)^\ul\big)\right)
=(\psi_n)^{\ul}\left((J_n(x_n))^\ul\right)
=\lim_{n \to \ul} \psi_n(J_n(x_n))
=\lim_{n \to \ul} \phi_n(x_n)
=(\phi_n)^{\ul}\big((x_n)^{\ul}\big).          
\end{align*}  
We infer that $(J_n)^\ul$ preserves the states. Hence $(J_n)^\ul$ is a $((\phi_n)^\ul,(\psi_n)^\ul)$-Markov map. Finally, for any element $(x_n)^\ul$ of $(M_{n})^\ul$ and any $(y_n)^\ul$ of $(N_{n})^\ul$, using \eqref{Adjoint} in the fourth equality, we have 
\begin{align*}
\MoveEqLeft
(\psi_n)^\ul\big((y_n)^\ul (J_n)^\ul((x_n)^\ul)\big)
=(\psi_n)^\ul\big((y_n)^\ul (J_n(x_n))^\ul)\big) \\       
&=(\psi_n)^\ul\big((y_n J_n(x_n))^\ul)\big)  
=\lim_{n \to \ul} \psi_n\big(y_n J_n(x_n)\big)
=\lim_{n \to \ul} \phi_n\big(\E_n(y_n) x_n\big)\\
&=(\phi_n)^\ul\big((\E_n(y_n) x_n)^\ul\big)
=(\phi_n)^\ul\big((\E_n(y_n))^\ul(x_n)^\ul\big).
\end{align*}  
Hence by unicity of $((J_n)^\ul)^*$, we obtain \eqref{Esperance-ul}.
%
%
\end{proof}

The following is probably folklore but we are unable to locate a proof in the literature. So we give a proof because of its importance in this paper.

\begin{prop}
\label{Prop-injection}
Let $M$ be a von Neumann algebra equipped with a normal faithful state $\varphi$. The map $\mathcal{I} \co M \to M^\ul$, $x \mapsto (x,x,\ldots)^\ul$ is a well-defined $\phi$-Markov $*$-monomorphism. Moreover, if we denote by $\E=\mathcal{I}^* \co M^\ul \to M$ the canonical normal faithful conditional expectation preserving the states associated with $\mathcal{I}$ then for any element $(x_n)^\ul$ of $M^\ul$, we have
\begin{equation}
\label{Esperance2-formmule}
\E\big((x_n)^\ul\big)
=\w^*\text{-}\lim_{n \to \ul} x_n.	
\end{equation}
\end{prop}

\begin{proof}
 First, if $x \in M$ then we will show that $(x,x,\ldots)$ belongs to $\mathcal{M}^{\ul}(M,\phi)$. If $(y_n)$ is an element of $\mathcal{I}_{\ul}(M,\phi)$, it suffices to check that $(xy_n) \in \mathcal{I}_{\ul}(M,\phi)$ and that $(y_n x) \in \mathcal{I}_{\ul}(M,\phi)$. By the discussion following \eqref{Norms}, this is equivalent to the convergence of the sequences $(xy_n)$ and $(y_n x)$ along $\ul$ to 0 for the $\sigma$-strong* topology. Recall that by \cite[Lemma 2.5]{Tak1} the topology of $\sigma$-strong* topology coincide with the strong* topology on bounded sets. Using the separate continuity of the operator product for strong operator topology, it is not difficult to obtain the convergence. 
%
%
Hence $\mathcal{I}$ is well-defined. It is clearly a unital $*$-homomorphism. Using \eqref{Modular-ultraproduct}, for any $t \in \R$ and any $x \in M$, we have
\begin{align*}
\MoveEqLeft
 \mathcal{I} \circ  \sigma_t^{\phi}(x)         
=\big(\sigma_t^{\phi}(x),\sigma_t^{\phi}(x),\ldots\big)^\ul
=\sigma_t^{\phi^{\ul}}\big((x,x,\ldots)^\ul\big)
=\sigma_t^{\phi^{\ul}} \circ \mathcal{I}(x).
\end{align*}
Hence $\mathcal{I} \circ \sigma_t^{\phi}=\sigma_t^{\phi^{\ul}} \circ \mathcal{I}$, i.e. the map $\mathcal{I}$ commutes with the modular automorphism groups. Moreover, for any $x \in M$, we have
\begin{align*}
\MoveEqLeft
\phi^{\ul}\left(\mathcal{I}(x)\big)\right)
=\phi^{\ul}\left((x,x,\ldots)^\ul\right)
=\lim_{n \to \ul} \phi(x)
=\phi(x).          
\end{align*}  
We infer that $\mathcal{I}$ preserves the states. Hence the map $\mathcal{I}$ is a $(\phi,\phi^\ul)$-Markov map. By Lemma \ref{Lemma-injectivity}, we obtain the injectivity of $\mathcal{I}$.

Note that for any element $(x_n)$ of $\ell^{\infty}(\mathbb{N},M_n)$, the limit $\w^*\text{-}\lim_{n \to \ul} x_n$ exists. For any $y \in M$ and any element $(x_n)^\ul$ of the ultrapower $M^\ul$, using the fact that the product of a von Neumann algebra is separately weak* continuous in the last equality (e.g. see \cite[Proposition 2.7.4 (1)]{BLM}), we have
\begin{align*}
\MoveEqLeft
\phi^\ul\big((x_n)^\ul \mathcal{I}(y)\big)           
		=\phi^\ul\big((x_n)^\ul (y,y,\ldots)^\ul\big)
		=\phi^\ul\big((x_ny)^\ul\big)
		=\lim_{n \to \ul} \phi(x_ny)\\
		&=\phi\Big(\w^*\text{-}\lim_{n \to \ul} (x_n y)\Big)
		=\phi\Big(\Big(\w^*\text{-}\lim_{n \to \ul} x_n\Big)y\Big).
\end{align*}	
So by unicity of $J^*$, we obtain \eqref{Esperance2-formmule}.
\end{proof}

\paragraph{Convexity} A normed linear space $X$ is locally uniformly convex \cite[Definition 5.3.2]{Meg1} \cite[Definition 0.2]{Lov1} if for any $\epsi>0$ and any $x \in X$ with $\norm{x} =1$ there exists $\delta(\epsi,x)>0$ such that $\norm{y} = 1$ and $\frac{\norm{x+y}}{2} \geq 1-\delta(\epsi,x)$ imply $\norm{y-x} \leq \epsi$. 
 It is clear from the definition that uniform convexity implies local uniform convexity.

\paragraph{Semigroups of operators}
Let $X$ be a Banach space. Recall that a semigroup $(T_t)_{t \geq 0}$ of operators on $X$ is strongly continuous if for any $x \in X$ the map $t \mapsto T_t(x)$ is continuous from $\R^+$ into $X$. 

Let $X$ be a dual Banach space with predual $X_*$. We say that a semigroup $(T_t)_{t \geq 0}$ of bounded operators on $X$ is weak* continuous if the map $t \mapsto T_t$ is continuous from $\R$ into $\B(X)$ equipped with the point weak* topology, i.e. if the map $t \mapsto \big\langle y, T_t(x) \big\rangle_{X_*,X}$ is continuous on $\R^+$ for any $x \in X$ and any $y \in X_*$. 

\paragraph{Representations of groups and kernels}
Let $X$ be a Banach space. Let $\pi \co G \to \B(X)$ be a representation of a group $G$ on $X$. Then we say that $\pi$ is bounded when $\sup \big\{ \norm{\pi_t} : t \in G\big\} < \infty$.

We need some notions and results of the papers \cite{DLG1} and \cite{DLG2}. Recall that a non-empty subset $D$ of a semigroup\footnote{\thefootnote. A semigroup is a set supplied with an associative binary operation. Unfortunately, in this paper, we also use semigroups $(T_t)_{t \geq 0}$ indexed by $\R^+$. For such a semigroup, we require that $T_0=\Id$.} $\mathscr{S}$ is called a two-sided ideal \cite[page 65]{DLG1} \cite[page 318]{EFHN} if $\mathscr{S} D \subset D$ and if $D \mathscr{S} \subset D$. If $\mathscr{S}$ is a semigroup, the intersection of all the two-sided ideals of $\mathscr{S}$ is called the kernel of $\mathscr{S}$ \cite[page 66]{DLG1} \cite[page 318]{EFHN}. If $\mathscr{S}$ is a compact (Hausdorff) semitopological semigroup, that is a semigroup with a separately continuous semigroup operation, then it is known \cite[Theorem 2.3]{DLG1} \cite[Lemma 16.4]{EFHN} that its kernel is non-empty. 

Let $\pi \co G \to \B(X)$ be a (non-continuous) bounded representation of a topological group $G$ on a reflexive Banach space $X$. We denote by 
$$
X_c
=\{x \in X \ :\ t \mapsto \pi_t(x) \text{ is continuous from } G \text{ to } X \text{ equipped with the weak topology}\}
$$
the subspace of continuously translated elements of $X$ for the representation $\pi$, see \cite[Definition 2.1]{DLG2}. By \cite[Corollary 2.9]{DLG2}, if $G$ is locally compact then for any $x \in X_c$ the map $t \mapsto \pi_t(x)$ is continuous.

Let $\mathcal{V}_G(e)$ be the set of all neighbourhoods $V$ of the identity $e$ of $G$. We then set $\mathcal{S}^c(\pi)$ be the closure in the weak operator topology of the convex hull of $\bigcap_{V \in \mathcal{V}_G(e)} \ovl{\{\pi_t: t \in V \}}^{\mathrm{wo}}$ and called the convex semigroup of $\pi$ over the identity $e$ \cite[page 140]{DLG2}\footnote{\thefootnote. In \cite[page 140]{DLG2}, the convex semigroup over $e$ of a representation $U$ is denoted by $S(U)$.}. Then it is known \cite[Lemma 2.3]{DLG2} that $\mathcal{S}^c(\pi)$ endowed with the weak operator topology is a compact semitopological semigroup. The papers \cite{DLG1} and \cite{DLG2} give the following result which will be only used in Section \ref{sec:dilations-Lp}.

\begin{thm}
\label{DLG}
Let $X$ be a reflexive Banach space and $\pi \co G \to \B(X)$ be a (non-continuous) bounded representation of a commutative topological group $G$. Then the kernel $\mathcal{K}(\pi)$ of the convex semigroup $\mathcal{S}^c(\pi)$ of $\pi$ contains a unique idempotent $Q$ and $Q$ is a bounded projection of $X$ on $X_c$ with $Q\pi_t=\pi_tQ$ for any $t \in G$.
\end{thm}


\begin{proof}
First, by \cite[page 136]{DLG2} that any bounded (non-continuous) representation of $G$ on a reflexive Banach space is locally weakly almost periodic, that is the assumption written in \cite[page 139]{DLG2} of \cite[Section 2]{DLG2} is satisfied. Since $G$ is abelian, the remark before \cite[Lemma 2.4]{DLG2} says that the kernel of $\mathcal{K}(\pi)$ reduces to a single idempotent $Q$. By \cite[Lemma 2.4]{DLG2}, $Q$ is a bounded projection from $X$ onto $X_c$ with $Q\pi_t=\pi_tQ$ for any $t \in G$. 
\end{proof}

\paragraph{Accumulation points} Let $(y_i)_{i \in I}$ be a net in a topological space $Y$. An accumulation point of the net $(y_i)_{i \in I}$ is an element of the intersection $ \bigcap_{F \in \mathcal{F}} \ovl{F}$ where 
$$
\mathcal{F}
=\big\{F \subset X:\text{ there exists } i_0 \in I \text{ such that }\{y_{i} : i \geq i_0\} \subset F \big\}
$$ 
or equivalently a limit of some subnet of $(y_i)_{i \in I}$.

\section{Dilations of semigroups on von Neumann algebras}
\label{sec:Dilations-VN}

Suppose that $X$ is a dual Banach space $X$ with predual $X_*$. Recall that $\B(X)$ is a dual Banach space with the Banach space $X \hat{\ot} X_*$ as predual where $\hat{\ot}$ denotes the projective tensor product. Note that the weak* topology on $\B(X)$ is different from the point weak* topology. However, these topologies coincide on bounded subsets of $\B(X)$ by \cite[Lemma 7.2]{Pau}. We will often use this identification without saying it.

It is well-known that the space $\B_{\w^*}(X)$ of weak* continuous operator of $\B(X)$ is a semitopological semigroup with respect to the point weak* topology, see \cite[Exercise 1.12 page 251]{BJM} and \cite[Lemma 2.1]{BGKS}.
%
 
Let $G$ be a topological group and let $\pi \co G \to \B(X)$ be a (non-continuous) bounded representation on a dual Banach space $X$. We define the subspace
$$
X_{\w^*}
=\Big\{x \in X \ :\ t \mapsto \big\langle y,\pi_t(x)\big\rangle_{X_*,X} \text{ is continuous from } G \text{ to $\C$ for any $y$} \in X_*\Big\}
$$
of $X$ called subspace of weak* continuously translated elements of $X$. Recall that $\mathcal{V}_G(e)$ denotes the set of all neighbourhoods $V$ of the identity $e$ of $G$. We then set $\mathcal{S}^{\w^*}(\pi)$ to be the closure in the weak* topology of $\B(X)$ of the convex hull of $\bigcap_{V \in \mathcal{V}_G(e)} \ovl{\{\pi_t: t \in V \}}^{\w^*}$, endowed with the weak* topology where the closure is also taken for the weak* topology:
\begin{equation}
\label{Def-S-weak}
\mathcal{S}^{\w^*}(\pi)	
=\ovl{\Conv}^{\w^*}\bigcap_{V \in \mathcal{V}_G(e)} \ovl{\{\pi_t: t \in V \}}^{\w^*}.
\end{equation}
The following proposition is a weak* analogue of \cite[lemma 2.3]{DLG2}.

\begin{prop}
\label{prop-fixed-points=Xwstar}
Let $\pi \co G \to \B(X)$ be a bounded (non-continuous) representation of a topological group $G$ on a dual Banach space $X$. The set $X_{\w^*}$ consists of precisely those $x$ in $X$ which are fixed under all $T$ in $\bigcap_{V \in \mathcal{V}_G(e)} \ovl{\{\pi_t: t \in V \}}^{\w^*}$:
\begin{equation*}
X_{\w^*}
=\Bigg\{ x \in X \ \co \ T(x)=x \text{ for any $T \in \bigcap_{V \in \mathcal{V}_G(e)} \ovl{\{\pi_t: t \in V \}}^{\w^*}$} \Bigg\}.
\end{equation*}
\end{prop}

\begin{proof}
%
%
Consider $x \in X_{\w^*}$. If $y \in X_*$ then using the continuity of $t \mapsto \big\langle y,\pi_t(x)\big\rangle_{X_*,X}$ at the neutral element $e$, we see that for any $\epsi >0$ that there exists a neighbourhood $V_{\epsi,x,y}$ of the neutral element $e$ such that for any $t \in V_{\epsi,x,y}$
\begin{equation}
	\label{plus-petit-epsi}
\left|\big\langle y,\pi_t(x) \big\rangle_{X_*,X}-\langle y,x \rangle_{X_*,X}\right|
<\epsi.
\end{equation}
Let $T$ be an element of the closure $\overline{\{\pi_t \ :\ t \in V_{\epsi,x,y} \}}^{\w^*}$. There exists a net $\big(\pi_{t_i}\big)_{i \in I}$ with $t_i \in V_{\epsi,x,y}$ converging to $T$ in the weak* topology. For any $i \in I$, by \eqref{plus-petit-epsi}, we have 
$$
\left|\big\langle y,\pi_{t_i}(x) \big\rangle_{X_*,X}-\langle y,x \rangle_{X_*,X}\right|
<\epsi.
$$ 
Passing to the limit, we obtain
$$
\left|\big\langle y,T(x)-x \big\rangle_{X_*,X}\right|
=\left|\big\langle y,T(x) \big\rangle_{X_*,X}-\langle y,x \rangle_{X_*,X}\right|
<\epsi.
$$
Now, if $T_0 \in \bigcap_{V \in \mathcal{V}_G(e)} \ovl{\{\pi_t: t \in V \}}^{\w^*}$ then for any $\epsi>0$ and any $y \in X_*$ the element $T_0$ belongs to $\overline{\{\pi_t : t \in V_{\epsi,x,y} \}}^{\w^*}$. For any $y \in X_*$, we deduce that $|\langle y,T_0(x)-x \rangle_{X_*,X}| <\epsi$ for any $\epsi >0$ and thus
$$
\left|\big\langle y,T_0(x)-x \big\rangle_{X_*,X}\right|
=0.
$$
We conclude that $T_0(x)=x$. 

For the reverse inclusion, let $x \in X$ fixed by all elements of $\bigcap_{V \in \mathcal{V}_G(e)} \ovl{\{\pi_t: t \in V \}}^{\w^*} $, i.e. suppose that for any $T \in \bigcap_{V \in \mathcal{V}_G(e)} \ovl{\{\pi_t: t \in V \}}^{\w^*}$ we have $T(x)=x$. Consider a net $(t_{i})_{i \in I}$ in $G$ converging to the identity $e$. Since the representation $\pi$ is bounded, the subset $\ovl{\{\pi_t : t \in G\}}^{\w^*}$ is weak* compact by \cite[Corollary 2.6.19]{Meg1} (hence compact for the point weak* topology). Using the continuous map $\B(X) \to X$, $T \mapsto T(x)$, where the first space is equipped with the point weak* topology and the second with the weak* topology, we see that the subset $\ovl{\{\pi_t : t \in G\}}^{\w^*} \cdot x$ of $X$ is compact for the weak* topology.

Note that an accumulation point of the net $\big(\pi_{t_{i}}\big)_{i \in I}$ is an element of $\bigcap_{F \in \mathcal{F}} \ovl{F}^{\w^*}$ where 
$$
\mathcal{F}
=\big\{F \subset \B(X):\text{ there exists } i_0 \in I \text{ such that } \{\pi_{t_{i}} :  i \geq i_0\} \subset F \big\}.
$$ 
For any neighbourhood $V$ of the neutral element $e$ there exists $i_V$ such that $i \geq i_V$ implies $t_{i} \in V$ and thus $\pi_{t_{i}} \in \pi(V)$. Thus the set $\{\pi_{t_{i}} : i \geq i_V \}$ is included in $\{\pi_t : t \in V \}$. Then the set $\{\pi_t: t \in V \}$ belongs to $\mathcal{F}$. We deduce that
$$
\bigcap_{F \in \mathcal{F}} \ovl{F}^{\w^*} \subset \bigcap_{V \in \mathcal{V}_G(e)} \ovl{\{\pi_t: t \in V \}}^{\w^*}.
$$
We conclude that the net $\big(\pi_{t_{i}}\big)_{i \in I}$ can have accumulation points only in the intersection $\bigcap_{V \in \mathcal{V}_G(e)} \ovl{\{\pi_t: t \in V \}}^{\w^*}$. 

Now, it is not difficult to see that the net $\big(\pi_{t_{i}}(x)\big)_{i \in I}$ of $X$ can only have accumulation points in the weak* topology of in $\bigcap_{V \in \mathcal{V}_G(e)} \ovl{\{\pi_t: t \in V \}}^{\w^*}\cdot x =\{x\}$. For this, let $z \in X$ be an accumulation point of $\big(\pi_{t_{i}}(x)\big)_{i \in I}$ in the weak* topology of $X$. Given a neighbourhood $V$ of $e$, we find, for any weak* topology neighbourhood $W$ of $z$, some $i \in I$ such that
$$
t_i \in V \quad \text{and} \quad \pi_{t_{i}}(x) \in W.
$$
Hence we have
$$
\Big( \ovl{\{\pi_t : t \in V \}}^{\w^*} \cdot x\Big)\cap W \not= \emptyset.
$$
We infer that $z$ belongs to the weak* closure of $\ovl{\{\pi_t : t \in V \}}^{\w^*} \cdot x$. Since the subset $\ovl{\{\pi_t : t \in V \}}^{\w^*} \cdot x$ is closed as the continuous image of the compact set $ \ovl{\{\pi_t : t \in V \}}^{\w^*}$, we conclude that $z$ belongs to $\ovl{\{\pi_t : t \in V \}}^{\w^*} \cdot x$.

Lying in the weak* compact subset $\ovl{\{\pi_t : t \in G \}}^{\w^*} \cdot x$ of $X$, we infer that the net $\big(\pi_{t_{i}}(x)\big)_{i \in I}$ converges for the weak* topology to $x$. Consequently, the map $t \mapsto \pi_t(x)$ is weak* continuous at $t=e$, hence everywhere, completing the proof.
\end{proof}

%

%

The following result is a particular case of the combination of \cite[Theorem 1.2]{BGKS}, \cite[Proposition 5.5]{BGKS}, \cite[Remark 5.6]{BGKS} and \cite[Corollary 4.3]{BGKS} (and its proof). This is \cite[Theorem 2.4]{KuN} with some complements. 

\begin{thm}
\label{thm-Batkai}
Let $M$ be a von Neumann algebra equipped with a normal faithful state $\phi$. Let $\mathscr{S}$ be a semigroup of $\phi$-Markov maps. The closure $\ovl{\Conv \mathscr{S}}^{\w^*}$ of the convex hull $\Conv(\mathscr{S})$ of $\mathscr{S}$ in the weak* topology of $\B(M)$ is a compact semitopological semigroup and its kernel is a singleton $\{\E \}$ where $\E$ is a faithful normal conditional expectation $\E \co M \to M$ leaving $\phi$ invariant satisfying
$$
\ran \E 
= \big\{ x \in M : T(x)=x \ \text{ for any $T \in \mathscr{S}$ } \big\}.
$$
\end{thm}

\begin{proof}
By Lemma \ref{Lemma-prop-stabilite}, $\Conv \mathscr{S}$ is a bounded semigroup consisting of $\phi$-Markov maps. In particular, each map $T \co M \to M$ of $\Conv \mathscr{S}$ is a unital completely positive map, hence a Schwarz map by \cite[Proposition 3.3]{Pau}, i.e.:
$$
T(x)^*T(x) 
\leq T(x^*x), \qquad x \in M.
$$ 
By composition with the state $\varphi$, we deduce that for any $x \in M$
$$
\varphi\big(T(x)^*T(x) \big)
\leq \varphi\big(T(x^*x)\big)
=\varphi(x^*x),
$$
that is the assumption \cite[Theorem 3.2 (2)]{BGKS}. Consequently, by applying \cite[Theorem 3.2]{BGKS} with $\Conv \mathscr{S}$ instead of $\mathscr{S}$, we deduce that the closure $\ovl{\Conv \mathscr{S}}^{\w^*}$ of the convex hull $\Conv(\mathscr{S})$ of $\mathscr{S}$ in the point weak* topology of $\B(M)$ is a compact semitopological semigroup and that its kernel is a compact topological group for the point weak* topology whose the unit $P$ is the unique minimal projection of $\ovl{\Conv \mathscr{S}}^{\w^*}$. Note that the proof of \cite[Theorem 3.2]{BGKS} shows that each element of $\ovl{\Conv \mathscr{S}}^{\w^*}$ is weak* continuous and it is clear by a limit argument that an element of $\ovl{\Conv \mathscr{S}}^{\w^*}$ preserves the state. Hence $P \co M \to M$ is weak* continuous and preserves the state. Now the proof of \cite[Theorem 4.3]{BGKS} and its proof say that the kernel is equal to the singleton $\{P\}$ and that
$$
\ran P 
= \big\{ x \in M : T(x)=x \ \text{ for any $T \in \mathscr{S}$ } \big\}.
$$ 
By \cite[Remark 5.6]{BGKS}, the projection $P$ is faithful in the sense of \cite[page 19]{BGKS}. Then \cite[Proposition 5.5]{BGKS} says\footnote{\thefootnote. Note that \cite[Theorem 3.2]{BGKS} is a particular case of \cite[Theorem 1.2]{BGKS}.} that the weak* continuous projection $P$ leaving $\phi$ invariant is a (normal) faithful conditional expectation $\E \co M \to M$.
\end{proof}

The following lemma is a generalization of \cite[Lemma 3]{Fen1}. Thanks to the uniformly convexity of noncommutative $\L^p$-spaces \cite[Corollary 5.2]{PiX}, this lemma can be applied to noncommutative $\L^p$-spaces.
\begin{lemma}
\label{lemma-continuité-uniform-convexity}
Let $X$ be a Banach space and let $Y$ be a locally uniformly convex Banach space. Let $(T_t)_{t \geq 0}$ be a strongly continuous semigroup of contractions on $X$. Let $(U_t)_{t \in \Q}$ be a (non continuous) group of isometries on $X$ and $J \co X \to Y$ and $P \co Y\to X$ two contractions  such that $T_t=PU_tJ$ for any $t \in \Q^+$. If $x \in X$ then the map
\begin{equation*}
\begin{array}{cccc}
           &   \Q  &  \longrightarrow   &  Y  \\
           &   t   &  \longmapsto       &  U_t J(x)  \\
\end{array}
\end{equation*}
is continuous from $\Q$ to $Y$ with its norm topology.
\end{lemma}

\begin{proof}
Let $x \in X$ with $\norm{x}=1$. Since $\Id_{X}=T_0=PJ$, we have $\norm{J(x)}_{Y}=\norm{x}_{X}=1$. By the locally uniform convexity of $Y$, if $\epsi>0$ there exists $\delta(\epsi,x)>0$ such that if $y \in Y$ satisfies $\norm{y}_{Y} = 1$ and $\frac{\norm{y+J(x)}_Y}{2}_{} \geq 1-\delta(\epsi,x)$ we have 
\begin{equation}
\label{ine-50}
\norm{y-J(x)}_{Y} 
\leq \epsi.	
\end{equation}
Since $(U_t)_{t \in \Q}$ is a group of isometries, it suffices to show, for $x \in X$, the continuity from the right of the map $s \mapsto U_tJ(x)$ at $t=0$. Given $\epsi >0$, by the strong continuity of $(T_t)_{t \in \Q}$ there exists $\eta >0$ such that $0 \leq t \leq \eta$ implies $\norm{T_t(x)-x}_{X} \leq \delta(\epsi,x)$. Hence, for any $t \in \Q \cap [0, \eta]$ we have
\begin{align*}
  \bnorm{U_tJ(x)+J(x)}_{Y} &\geq \norm{PU_tJ(x)+PJ(x)}_{X}  
		= \norm{T_t(x)+x}_{X}
		= \bnorm{2x-(x-T_t(x))}_{X} \\
		&\geq  \norm{2x}_{X}- \norm{T_t(x)-x}_{X}
		\geq 2-\delta(\epsi,x)
		\geq 2-2\delta(\epsi,x).
\end{align*}
Hence $\frac{\norm{U_tJ(x)+J(x)}_{Y}}{2} \geq 1-\delta(\epsi,x)$. Since $\norm{U_tJ(x)}_{Y}=\norm{J(x)}_{Y}=\norm{x}_{X}=1$, by \eqref{ine-50}, we infer that $\norm{U_tJ(x)-J(x)}_{Y} \leq \epsi$.
\end{proof}

The following lemma is a variant of the above lemma and is a key lemma. The proof uses mysteriously some noncommutative $\L^p$-spaces for $1< p< \infty$.

\begin{lemma}
\label{prop-weak-continuity-not-destroyed}
Let $M,$ and $N$ be von Neumann algebras equipped with normal faithful states $\phi$ and $\psi$. Let $(T_t)_{t\geq 0}$ be a weak* continuous semigroup of $\phi$-Markov maps on $M$. Let $(U_t)_{t \in \Q}$ be a group of $*$-automorphisms of $N$ leaving $\psi$ invariant and $J \co M \to N$ a $(\phi,\psi)$-Markov $*$-monomorphism such that $T_t=\E U_tJ$ for any $t \in \Q^+$ where $\E \co N \to M$ is the canonical faithful normal conditional expectation preserving the states associated with $J$. For any $x \in M$ and any $y \in L^1(N)$, the map
\begin{equation*}
\begin{array}{cccc}
          &  \mathbb{Q} &  \longrightarrow   & \C   \\
          &  t &  \longmapsto       &  \big\langle y,U_tJ(x)\big\rangle_{\L^1(N),N} \\
\end{array}
\end{equation*}
is continuous. 
\end{lemma}

\begin{proof}
%
We fix $1< p< \infty$. The semigroup $(T_t)_{t \geq 0}$ induces a strongly continuous semigroup $(T_{t,p})_{t \geq 0}$ of contractions on $\L^p(M)$ and the semigroup $(U_t)_{t \in \Q}$ induces a group of isometries $(U_{t,p})_{t \in \Q}$ on $\L^p(N)$. Moreover $J$ induces an isometric embedding $J_p$ of $\L^p(M)$ into $\L^p(N)$ and $\E$ a contractive map $\E_p$ from $\L^p(N)$ onto $\L^p(M)$. For any $x \in \L^p(M)$, by Lemma \ref{lemma-continuité-uniform-convexity}, the map $t \mapsto U_{t,p} J_p(x)$ is continuous from $\Q$ into $\L^p(N)$ with its norm topology. Let $t_0 \in \Q$ and let $D_\psi \in \L^1(N)$ be the density operator of $\psi$. Note that $D_\psi^{\frac{1}{2p}}$ belongs to $\L^{2p}(N)$. For any $z \in \L^{p^*}(N)$ and any $x \in N$, we have
\begin{align*}
\MoveEqLeft
  \big\langle D_\psi^{\frac{1}{2p}}zD_\psi^{\frac{1}{2p}},U_tJ(x)\big\rangle_{\L^1(N),N}
		=\big\langle z,D_\psi^{\frac{1}{2p}}U_tJ(x)D_\psi^{\frac{1}{2p}}\big\rangle_{\L^{p^*}(N),\L^p(N)}\\
		&=\big\langle z,U_{t,p}J_p(x)\big\rangle_{\L^{p^*}(N),\L^p(N)}
		\xra[t \to t_0]{} \big\langle z,U_{t_0,p}J_p(x)\big\rangle_{\L^{p^*}(N),\L^p(N)}\\
		&=\big\langle z,D^{\frac{1}{2p}}_\psi U_{t_0}J(x)D^{\frac{1}{2p}}_\psi\big\rangle_{\L^{p^*}(N),\L^p(N)}
		=\big\langle D^{\frac{1}{2p}}_\psi zD^{\frac{1}{2p}}_\psi,U_{t_0}J(x)\big\rangle_{\L^{1}(N),N}.
\end{align*}
Recall that $D_\psi^{\frac{1}{2p}}\L^{p^*}(N)D_\psi^{\frac{1}{2p}}$ is norm dense in the Banach space $\L^1(N)$. Now, with a $\frac{\epsi}{3}$-argument, it is not difficult to complete this proof. 
\end{proof}

Now we can prove our first main result. We use a similar strategy to the one of Fendler \cite{Fen1}.  However, the method of \cite{Fen1} does not apply identically to our context. We will use some results from the papers \cite{DLG1} and \cite{DLG2}.

\begin{thm}
\label{thm-dilation-semigroup-weak-star2}
let $M$ be a von Neumann algebra equipped with a normal faithful state $\phi$. Let $(T_t)_{t \geq 0}$ be a weak* continuous semigroup of factorizable $\phi$-Markov map on $M$. Then the semigroup $(T_t)_{t \geq 0}$ is dilatable. 
\end{thm}

\begin{proof}
For a finite set $B \subset \Q$ let $U_B=\{ n \in \mathbb{N} : nt \in \mathbb{Z} \text{ for any } t \in B \}$\footnote{\thefootnote. Roughly speaking, the set $U_B$ consists of the common multiples of the denominators of the rationals of $B$.  }. Then the set of all sets $ \{ U_B : B \subset \mathbb{Q},\ B \text{ finite} \}$ is closed under finite intersections\footnote{\thefootnote. Note that $I_B=\{ n \in \mathbb{Z} : nt \in \mathbb{Z} \text{ for any } t \in B \}$ is an ideal of $\Z$. Hence we can write $I_B=n_B\mathbb{Z}$. We deduce that $U_B=n_B \mathbb{N}$. Finally $U_B \cap U_{B'}=U_{\big\{\frac{1}{\mathrm{LCM}(n_B,n_{B'})}\big\}}$.} and thus constitutes the basis of some filter $\mathcal{F}$ which is contained in some ultrafilter $\ul$.

Using \cite[Theorem 4.4]{HaM}, for any integer $n \geq 0$, we note that the operator $T_{\frac{1}{n}}\co M \to M$ is dilatable. This means that there exist a von Neumann algebra $N_{\frac{1}{n}}$ equipped a normal faithful state $\varphi_{\frac{1}{n}}$, a $*$-automorphism $S_{\frac{1}{n}}$ of $N_{\frac{1}{n}}$ leaving $\varphi_{\frac{1}{n}}$ invariant and a $(\phi,\varphi_{\frac{1}{n}})$-Markov $*$-monomorphism $J_{\frac{1}{n}} \co M \to N_{\frac{1}{n}}$ such that
\begin{displaymath}
\big(T_{\frac{1}{n}}\big)^k
=\E_{\frac{1}{n}} \big(S_{\frac{1}{n}}\big)^{k} J_{\frac{1}{n}}, \qquad  k \geq 0,
\end{displaymath}
where $\E_{\frac{1}{n}}=(J_{\frac{1}{n}})^* \co N_{\frac{1}{n}} \to M$ is the canonical $\varphi_{\frac{1}{n}}$-preserving normal faithful conditional expectation associated with $J_{\frac{1}{n}}$. For $t \in \mathbb{Q}$, we define the operator $S_{\frac{1}{n},t} \co N_{\frac{1}{n}} \to N_{\frac{1}{n}}$ by 
\begin{displaymath} 
S_{\frac{1}{n},t} 
=\left\{ \begin{array}{r@{\quad \quad}l}
 (S_{\frac{1}{n}})^{nt}  &  \text{if} \quad nt \in \mathbb{Z} \\ 
\Id_{ N_{\frac{1}{n}}} & \mbox{if} \quad nt \notin \mathbb{Z}
\end{array}. \right.  
\end{displaymath}
If $B = \{t_1,\dots,t_k \} \subset \mathbb{Q^+} $ is a finite subset, then for $t \in B$ and $n \in U_B$ we have $nt \in \Z$ and thus 
\begin{equation}
\label{commute}
	T_t=\big(T_{\frac{1}{n}}\big)^{nt}
=\E_{\frac{1}{n}}  \big(S_{\frac{1}{n}}\big)^{nt} J_{\frac{1}{n}}
=\E_{\frac{1}{n}} S_{\frac{1}{n},t}  J_{\frac{1}{n}},
\end{equation}
i.e. the following diagram commutes.
$$
 \xymatrix @R=1cm @C=2cm{
    N_{\frac{1}{n}}      \ar[r]^{S_{\frac{1}{n},t}}  &N_{\frac{1}{n}} \ar[d]^{\mathbb{E}_{\frac{1}{n}}} \\
    M    \ar[r]_{T_{t}}\ar[u]^{J_{\frac{1}{n}}}&  M
}
$$

We consider the following ultraproducts of von Neumann algebras
$$
M^\ul
=(M,\phi)^\ul 
\quad \text{ and } \quad 
\widetilde{N}
=\big(N_{\frac{1}{n}},\varphi_{\frac{1}{n}}\big)^\ul.
$$
We equip $\widetilde{N}$ with the normal faithful state $\varphi=\big(\varphi_{\frac{1}{n}}\big)^\ul$. Using Proposition \ref{Prop-injection}, we can consider the canonical inclusion $\mathcal{I} \co M \to M^\ul$ $x \mapsto (x,x,\ldots)^\ul$ which is a $(\phi,\phi^\ul)$-Markov $*$-monomorphism and the associated normal faithful conditional expectation $\E \co M^\ul \to M$, $(x_n)^\ul \mapsto \w^*\text{-}\lim_{n \to \ul} x_n$. Using Proposition \ref{Prop-ultramap}, we can introduce the operators
\begin{equation}
	\label{operators-1}
\widetilde{J}
= \big(J_{\frac{1}{n}}\big)^\ul \mathcal{I},\qquad 
\widetilde{S}_{t} = \big(S_{\frac{1}{n},t}\big)^\ul, 
\quad t \in \mathbb{Q}.	
\end{equation}
By composition the map $\widetilde{J} \co M \to \widetilde{N}$ is a $(\phi,\varphi)$-Markov $*$-monomorphism. For any $t \in \Q$, note also that the map $\widetilde{S_t} \co \widetilde{N} \to \widetilde{N}$ is a $*$-automorphism of $\widetilde{N}$ leaving $\varphi$ invariant. Let $\widetilde{\E}=\widetilde{J}^* \co \widetilde{N} \to M$ be the canonical faithful normal conditional expectation associated with $\widetilde{J}$. 
We have 
\begin{equation}
	\label{operator-2}
\widetilde{\E}
=\big(\widetilde{J}\big)^*
=\Big(\big(J_{\frac{1}{n}}\big)^\ul \mathcal{I}\Big)^*
=\mathcal{I}^*\Big(\big(J_{\frac{1}{n}}\big)^\ul\Big)^*
=\E \big(\E_{\frac{1}{n}}\big)^\ul.	
\end{equation}
Let us check that the map
\begin{equation*}
\begin{array}{cccc}
   \widetilde{S}   \co &  \mathbb{Q}   &  \longrightarrow   &  \B\big(\widetilde{N}\big)  \\
           &  t   &  \longmapsto       & \widetilde{S}_t   \\
\end{array}
\end{equation*}
is a representation and that it defines a dilation of the semigroup $(T_t)_{t \in \mathbb{Q}^+}$. Suppose $t, t' \in \mathbb{Q}$ and $x=(x_n)^\ul \in \widetilde{N}$. If $n \in U_{\{t,t'\}}$ (i.e. for $n$ sufficiently large) then we have $nt,nt' \in \Z$ and $n(t+t')=nt+nt' \in \Z$. Then we obtain
$$
S_{\frac{1}{n},t+t'}(x_n)   
=\big(S_{\frac{1}{n}}\big)^{n(t+t')}(x_n)
=\big(S_{\frac{1}{n}}\big)^{nt+nt'}(x_n)
=S_{\frac{1}{n}}^{nt}\big(S_{\frac{1}{n}}^{nt'}(x_n)\big)
=S_{\frac{1}{n},t}\big(S_{\frac{1}{n},t'}(x_n)\big).
$$
We have $\big(S_{n,t+t'}(x_n)\big)^\ul=\big(S_{n,t}\big(S_{n,t'}(x_n)\big)\big)^\ul$
and thus 
$$
\widetilde{S}_{t+t'}\big((x_n)^\ul\big)
=\widetilde{S}_{t}\widetilde{S}_{t'}\big((x_n)^\ul\big). 
$$
Moreover, for any $t \in \mathbb{Q}^+$ and any $x \in M$, using \eqref{operators-1} and \eqref{operator-2} in the first equality, we have
\begin{align*}
\MoveEqLeft
  \widetilde{\E}\widetilde{S_t}\widetilde{J}(x)
		=\E \big(\E_{\frac{1}{n}}\big)^\ul\big(S_{\frac{1}{n},t}\big)^\ul\big(J_{\frac{1}{n}}\big)^\ul \mathcal{I}(x)
		=\E \big(\E_{\frac{1}{n}}\big)^\ul\big(S_{\frac{1}{n},t}\big)^\ul\big(J_{\frac{1}{n}}\big)^\ul \mathcal{I}(x)
		=\w^*\text{-}\lim_{n \to \ul} \E_{\frac{1}{n}} S_{\frac{1}{n},t}J_{\frac{1}{n}}(x).
\end{align*}
By (\ref{commute}), if $n \in U_{\{t\}}$, we have $\E_{\frac{1}{n}} S_{\frac{1}{n},t}J_{\frac{1}{n}}(x)=T_t(x)$. We deduce that  
\begin{equation}
	\label{Grosse-dilation}
\widetilde{\E}\widetilde{S_t}\widetilde{J}
=T_t, \qquad t \in \mathbb{Q}^+.	
\end{equation}
Recall that the weak* topology on $\B(\widetilde{N})$ is finer than the point weak* topology. Thus, using Lemma \ref{Lemma-prop-stabilite}, we see that each element of $\bigcap_{V \in \mathcal{V}_{\mathbb{Q}}(0)} \ovl{\{\widetilde{S}_t: t \in V \}}^{\w^*}$ is a $\varphi$-Markov map, in particular weak* continuous. 

Using the bounded representation $\widetilde{S} \co \mathbb{Q} \to \B(\widetilde{N})$, $t \mapsto \widetilde{S}_t$, we can use the notation $\mathcal{S}^{\w^*}(\widetilde{S})$ defined in \eqref{Def-S-weak}.

\begin{lemma}
\label{Lemma-two-semigroups}
The sets $\bigcap_{V \in \mathcal{V}_{\mathbb{Q}}(0)} \ovl{\{\widetilde{S}_t: t \in V \}}^{\w^*}$ and $\mathcal{S}^{\w^*}(\widetilde{S})$ are semigroups for the product of operators.
\end{lemma}

\begin{proof} 
Let $T$ and $R$ be elements of $\bigcap_{V \in \mathcal{V}_{\mathbb{Q}}(0)} \ovl{\{\widetilde{S}_t: t \in V \}}^{\w^*}$. Let $U$ be a neighbourhood of $R$ for the weak* topology. For any neighbourhood $V$ of $e$, we have $U \cap \{\widetilde{S}_t: t \in V \} \not= \emptyset$. Hence there exists $t_{V,U} \in V$ such that $\widetilde{S}_{t_{V,U}} \in U$. The net\footnote{\thefootnote. Declare that $(V_1,U_1) \preceq (V_2,U_2)$ if $V_2 \subset V_1$ and $U_2 \subset U_1$.} $(t_{V,U})$ converges\footnote{\thefootnote. Let $V_0$ be a neighbourhood of $e$. Choose a neighbourhood $U_0$ of $R$ for the weak* topology. Then for any $(V,U)$ such that $(V,U) \succeq (V_0,U_0)$ we have $t_{V,U} \in V \subset V_0$.} to $e$ in $G$ and the net $(\widetilde{S}_{t_{V,U}})_{}$ converges\footnote{\thefootnote.  Let $U_0$ be a neighbourhood of $R$ for the weak* topology. Choose a neighbourhood $V_0$ of $e$. Then for any $(V,U)$ such that $(V,U) \succeq (V_0,U_0)$ we have $\widetilde{S}_{t_{V,U}} \in U \subset U_0$.} to $R$ in the weak* topology. 

Let $V$ be a neighbourhood of the neutral element $e$ of $G$. Using \cite[Theorem 4.5]{HeR1}, choose a neighbourhood $W$ of $e$ such that $W^2 \subset V$. Note that $T \in \ovl{\{\widetilde{S}_t: t \in W \}}^{\w^*}$. We have
$$
\{\widetilde{S}_t: t \in W \} \cdot \{\widetilde{S}_t: t \in W \} 
\subset 
\{\widetilde{S}_t : t \in V \}
\subset \ovl{\{\widetilde{S}_t: t \in V \}}^{\w^*}.
$$
If $V' \subset W$, i.e. if $(W,U) \preceq (V',U)$, we have $t_{V',U} \in V' \subset W$. Recalling that the multiplication of operators is separately continuous in the point weak* topology on the subspace of weak* continuous operators, we obtain\footnote{\thefootnote. Here, it is crucial that each operator of $\ovl{\{\widetilde{S}_t: t \in W \}}^{\w^*}$ is weak* continuous.}
$$
T \cdot\widetilde{S}_{t_{V',U}} \in \ovl{\{\widetilde{S}_t: t \in W \}}^{\w^*} \cdot \{\widetilde{S}_t: t \in W \}\subset \ovl{\{\widetilde{S}_t: t \in V \}}^{\w^*}.
$$ 
Passing to the limit, we deduce that $T R \in \ovl{\{\widetilde{S}_t: t \in V \}}^{\w^*}$ for any neighbourhood $V$ of $e$. Hence $TR$ belongs to the set $\bigcap_{V \in \mathcal{V}_\mathbb{Q}(0)} \ovl{\{\widetilde{S}_t: t \in V \}}^{\w^*}$, i.e. this latter set is a semigroup. Consequently, the convex hull of $\bigcap_{V \in \mathcal{V}_{\mathbb{Q}}(0)} \ovl{\{\widetilde{S}_t: t \in V \}}^{\w^*}$ is\footnote{\thefootnote. Note that $((1-t)x+ty)((1-t')x'+t'y')=tt'xx''+t(1-t')xy'+(1-t)t'yx'+(1-t)(1-t')yy'$ and that $tt'+t(1-t')+(1-t)t'+(1-t)(1-t')=1$.} also a semigroup. Separate continuity of multiplication shows by a two-step argument that the same is true of the point weak* closure $\mathcal{S}^{\w^*}(\pi)$. 
\end{proof}

From Theorem \ref{thm-Batkai} with $\mathscr{S}
=\bigcap_{V \in \mathcal{V}_{\mathbb{Q}}(0)} \ovl{\{\widetilde{S}_t: t \in V \}}^{\w^*}$, we deduce that the kernel of the weak* closure $\mathcal{S}^{\w^*}(\widetilde{S})$ of the convex hull of $\bigcap_{V \in \mathcal{V}_{\mathbb{Q}}(0)} \ovl{\{\widetilde{S}_t: t \in V \}}^{\w^*}$ is a singleton $\{\E_{\w^*}\}$ where $\E_{\w^*} \co \widetilde{N} \to \widetilde{N}$ is a faithful normal conditional expectation preserving $\varphi$ satisfying
$$
\ran \E_{\w^*}
=\Big\{ x \in \widetilde{N} : T(x)=x \ \text{ for any $T \in \mathscr{S}$ } \Big\}.
$$
By Proposition \ref{prop-fixed-points=Xwstar}, the subspace $\widetilde{N}_{\w^*}$ of weak* continuously translated elements of $\widetilde{N}$ of the representation $\Q \to \B\big(\widetilde{N}\big)$, $t \mapsto \widetilde{S}_t$ is equal to the fixed point subspace of $\mathscr{S}$:
$$
\widetilde{N}_{\w^*}
=\Big\{ x \in \widetilde{N} : T(x)=x \ \text{ for any $T \in \mathscr{S}$ } \Big\}.
$$ 
Hence the von Neumann algebra $\ran \E_{\w^*}$ is equal to $\widetilde{N}_{\w^*}$. Now, for any $t \in \mathbb{Q}$, we will prove the following lemma.

\begin{lemma}
\label{Lemma-stab}
For any $t \in \mathbb{Q}$, we have
$$
\widetilde{S_t}\E_{\w^*}
=\E_{\w^*}\widetilde{S_t}.
$$ 
\end{lemma}


\begin{proof}
For any $t \in \mathbb{Q}$, using the weak* continuity of $\widetilde{S_t}$ we have
\begin{align*}
\MoveEqLeft
  \widetilde{S_t}\left( \bigcap_{V \in \mathcal{V}_{\mathbb{Q}}(0)} \ovl{\{\widetilde{S}_s: s \in V \}}^{\w^*}\right) \widetilde{S_t}^{-1}
		=\bigcap_{V \in \mathcal{V}_{\mathbb{Q}}(0)} \widetilde{S_t} \ovl{\{\widetilde{S}_s: s \in V \}}^{\w^*} \widetilde{S_t}^{-1}\\
		&= \bigcap_{V \in \mathcal{V}_{\mathbb{Q}}(0)}\ovl{\widetilde{S_t} \{\widetilde{S_s} : s \in V \} \widetilde{S_t}^{-1}}^{\w^*}
		=\bigcap_{V \in \mathcal{V}_{\mathbb{Q}}(0)} \ovl{\{\widetilde{S_s}: s \in V \}}^{\w^*}.
\end{align*}
This implies that
$$
\widetilde{S_t}\bigg(\Conv \bigcap_{V \in \mathcal{V}_{\mathbb{Q}}(0)} \ovl{\{\widetilde{S}_s: s \in V \}}^{\w^*}\bigg)\widetilde{S_t}^{-1}
\subset \Conv \bigcap_{V \in \mathcal{V}_{\mathbb{Q}}(0)} \ovl{\{\widetilde{S}_s: s \in V \}}^{\w^*}
$$
and finally by using again the weak* continuity of $\widetilde{S_t}$
$$
\widetilde{S_t} \mathcal{S}^{\w^*}(\widetilde{S}) \widetilde{S_t}^{-1} 
=\mathcal{S}^{\w^*}(\widetilde{S}).
$$
So the map $T \mapsto \widetilde{S_t} T \widetilde{S_t}^{-1}$ induces an automorphism of the semigroup $\mathcal{S}^{\w^*}(\widetilde{S})$. But any automorphism of the semigroup $\mathcal{S}^{\w^*}(\pi)$ preserves the least ideal, i.e. the kernel $\{\E_{\w^*}\}$ of $\mathcal{S}^{\w^*}(\widetilde{S})$. We deduce that $\widetilde{S_t} \E_{\w^*} \widetilde{S_t}^{-1}=\E_{\w^*}$ for any $t \in \mathbb{Q}$.
\end{proof}

By \cite[Theorem 2.22]{AbA}, we deduce that $\ran \E_{\w^*}$ is invariant under the operator $\widetilde{S_t}$ for any $t \in \Q$. By Proposition \ref{prop-weak-continuity-not-destroyed}, the range $\Ran (\widetilde{J})$ of the map $\widetilde{J} \co M \to \widetilde{N}$ is contained in the subspace $\widetilde{N}_{\w^*}$ of continuously translated elements of $\widetilde{N}$ of the representation $\widetilde{S} \co \Q \to \B\big(\widetilde{N}\big)$, $t \mapsto \widetilde{S}_t$. Now, we consider the von Neumann algebra $N=\widetilde{N}_{\w^*}$ equipped with the restriction $\psi$ of the normal state $\varphi$ and for any $t \in \mathbb{Q}$ we introduce the $*$-automorphism preserving the states
$$
U_t
=\widetilde{S}_{t}|\widetilde{N}_{\w^*} \co \widetilde{N}_{\w^*} \to \widetilde{N}_{\w^*}.
$$
Note that the map $U \co \mathbb{Q} \to \B(N)$, $t \mapsto U_t$ is a representation and is continuous where $\B(N)$ is equipped with the point weak* topology. Finally, we let $J \co M \to \widetilde{N}_{\w^*}=N$ be the canonical $*$-monomorphism which is a $(\phi,\psi)$-Markov map and $\E=J^* \co N \to M$ be the associated normal faithful conditional expectation. 
$$
 \xymatrix @R=1cm @C=2cm{
\widetilde{N}\ar[r]^{\widetilde{S}_{t}}\ar@/^2pc/[r]^{\E_{\w^*}}&\widetilde{N}\ar@/^3pc/[dd]^{\widetilde{\E}}\\
    \widetilde{N}_{\w^*}=N   \ar@{^{(}->}[u]   \ar[r]^{U_t=\widetilde{S}_{t}|\widetilde{N}_{\w^*}}  &N=\widetilde{N}_{\w^*} \ar[d]^{\E}\ar@{^{(}->}[u] \\
    M    \ar[r]_{T_{t}}\ar[u]^{J} \ar@/^3pc/[uu]^{\widetilde{J}}&  M
}
$$
Using \eqref{Grosse-dilation}, it is (really) not difficult to see that for any $t \in \mathbb{Q}^+$ we have
$$
T_t
=\E U_tJ,\qquad t \in \mathbb{Q}^+.
$$
By point weak* continuity, the map $U \co \mathbb{Q} \to \B(N)$, $t \mapsto U_t$ can be extended to a continuous map $U \co \R \to \B(N)$ where $\B(N)$ is equipped the point weak* topology. Consider some fixed $t \in \R$. There exists a sequence $(t_k)$ of elements of $\Q$ which converges to $t$. Since $U_{t_k} \co N \to N$ is a homomorphism, for any $x,x' \in N$ and any $y \in N_*$, we have
$$
\big\langle y,U_{t_k}(xx') \big\rangle_{N_*,N}
=\big\langle y, U_{t_k}(x)U_{t_k}(x') \big\rangle_{N_*,N}.
$$
By letting $k$ tend to infinity, we deduce that $U_{t}(xx')=U_{t}(x)U_{t}(x')$. Similarly, for any $x \in N$ and any $t \in \R$, we prove that $U_t(x^*)=U_t(x)^*$, the linearity of $U_t$ and that $U_t$ preserves the state $\psi$. For any $x \in N$, any $y \in N_*$ and any $t,t' \in \mathbb{Q}$ we have
$$
\big\langle y,U_{t+t'}(x) \big\rangle_{N_*,N}
=\big\langle y, U_tU_{t'}(x) \big\rangle_{N_*,N}.
$$
It is not difficult to deduce by approximation that $U \co \R \to \B(N)$ is a representation, i.e. that for any $t,t' \in \R$ we have 
$$
U_{t+t'}
=U_t U_{t'}.
$$
We deduce that $U_tU_{-t}=U_{-t}U_t=U_0=\Id_{N}$. Hence each $U_t$ is a $*$-automorphism of $N$. We conclude that $t \mapsto U_t$ defines a weak* continuous group $(U_t)_{t \in \R}$ of $*$-automorphisms of $N$ leaving $\psi$ invariant. 

For any $x \in M$ and any $y \in M_*$, we obtain
$$
\big\langle y,T_t(x) \big\rangle_{M_*,M}
=\big\langle y,\E U_tJ(x) \big\rangle_{M_*,M} \qquad t \in \R^+,
$$
since both sides are continuous functions of $t \in \R^+$ and the above equality is valid for the dense subset $\mathbb{Q}^+$ of $\R^+$. We conclude that
$$
T_t
=\E U_tJ,\qquad t \in \R^+.
$$
\end{proof}

\begin{rk}
\label{Remarque-finite} 
It is obvious that if $M$ is a von Neumann algebra equipped with a faithful finite normal trace then $N$ is also equipped with a faithful finite normal trace. See \cite[Section 6.1]{AnH} for related things.
\end{rk}

\begin{rk}
\label{Remarque-QWEP} 
We refer to \cite{AHW}, \cite{AnH}, \cite{CL} and \cite{Oza} for QWEP von Neumann algebras. 
We say \cite[Definition 1.2]{Arh2} that a $\phi$-Markov map $T \co M \to M$ is QWEP-factorizable if the definition of factorizability of the introduction is satisfied with a QWEP von Neumann algebra $N$. Similarly, we say \cite[Definition 1.3]{Arh2} that a weak* continuous semigroup $(T_t)_{t \geq 0}$ of $\phi$-Markov maps is QWEP-dilatable if the definition \ref{Def QWEP dilatable} is satisfied with a QWEP von Neumann algebra $N$. It is easy to see that if each operator $T_t$ is QWEP-factorizable then the semigroup $(T_t)_{t \geq 0}$ is QWEP-dilatable. Indeed, for any integer $n$ the proof of \cite[Theorem 4.4]{HaM} gives a QWEP von Neumann algebra $N_{\frac{1}{n}}$ (use \cite[Proposition 4.1 (ii)b and (iii)]{Oza}). Now, if each $N_{\frac{1}{n}}$ has $\QWEP$, then by \cite[Lemma 4.3]{AHW} the ultraproduct $\widetilde{N}=(N_{\frac{1}{n}},\varphi_{\frac{1}{n}})^\ul$ has also $\QWEP$. Finally, we conclude that $N=\widetilde{N}_{\w^*}$ has QWEP by \cite[Proposition 4.1 (ii)]{Oza}. This remark is useful for applications of dilations to the theory \cite{Jun2} of vector-valued noncommutative $\L^p$-spaces associated to QWEP von Neumann algebras.
\end{rk}

\begin{rk}
\label{Remarque-hyperfinite} 
Let $M$ be a von Neumann algebra equipped with a normal faithful state $\phi$. Recall that we say \cite[Definition 1.2]{Arh2} that a $\phi$-Markov map $T \co M \to M$ is hyper-factorizable if the definition of factorizability of the introduction is satisfied with a hyperfinite von Neumann algebra $N$. We say that a weak* continuous semigroup $(T_t)_{t \geq 0}$ of $\phi$-Markov maps on $M$ is hyper-dilatable if the definition \ref{Def QWEP dilatable} is satisfied with a hyperfinite von Neumann algebra $N$. Then the following open question is natural. Does every weak* continuous semigroup $(T_t)_{t \geq 0}$ of hyper-factorizable $\phi$-Markov maps on $M$ is hyper-dilatable ? 
By \cite[Proposition 5.5]{Arh1}, the answer is positive for all weak* continuous semigroups of $\tr$-Markov selfadjoint Schur multipliers on $\B(\ell^2_n)$.
\end{rk}

\section{Concrete dilations on von Neumann algebras}
\label{Concrete}

A default of our construction is its non-constructivist nature. Hence a natural problem is to find other \textit{concrete} dilations of particular weak* continuous semigroups of factorizable $\phi$-Markov map. For example, \cite[Proposition 5.5]{Arh1} and \cite{Arh4} describe a concrete dilation for semigroups of selfadjoint Schur multipliers. The $q$-Ornstein-Uhlenbeck semigroup has a obvious dilation, essentially contained in the proof of \cite[Theorem 9.4]{JMX}. In the sequel, we will give other natural dilations of some classical semigroups.

\paragraph{Poisson semigroup on $\R^n$} 
If $\omega_n=\frac{2\pi^{\frac{n}{2}}}{\Gamma(\frac{n}{2})}$ is the usual area of the unit ball of $\R^n$, we can consider the Poisson kernel \cite[page 93]{Fol1} $\p_{\R^n,t}(x)=\frac{2t}{\omega_{n+1}(t^2+|x|^2)^{\frac{n+1}{2}}}$ where $t>0$ and $x \in \R^n$ and where $|\cdot|$ denotes the standard Euclidean norm on $\R^n$. We denote by $(\P_{\R^n,t})_{t \geq 0}$ the Poisson semigroup on $\L^\infty(\R^n)$. For any $t > 0$ and any $f \in \L^\infty(\R^n)$, we have for almost all $x \in \R^n$
$$
\big[\P_{\R^n,t}(f)\big](x)
=(\p_{\R^n,t}*f)(x)
=\int_{\R^n} \frac{2t}{\omega_{n+1}(t^2+|y|^2)^{\frac{n+1}{2}}}f(x-y) \d \mu_{\R^n}(y).
$$
Using the change of variables $y=tu$ in the first equality, we see that for any $t >0$ and almost all $x \in \R^n$
\begin{align}
\MoveEqLeft
\label{Change} 
\big[\P_{\R^n,t}(f)\big](x)     
		=\int_{\R^n} \frac{2t}{\omega_{n+1}(t^2+|tu|^2)^{\frac{n+1}{2}}}f(x-tu) t^n\d \mu_{\R^n}(u)\\
		&=\int_{\R^n} \frac{2}{\omega_{n+1}(1+|u|^2)^{\frac{n+1}{2}}}f(x-tu) \d \mu_{\R^n}(u)
		=\int_{\R^n} f(x-ty) \frac{2}{\omega_{n+1}(1+|y|^2)^{\frac{n+1}{2}}}\d \mu_{\R^n}(y). \nonumber
\end{align}
With the function $g \co \R^n \to \R$, $y \mapsto \frac{2}{\omega_{n+1}(1+|y|^2)^{\frac{n+1}{2}}}$, the computation of \cite[page 93]{Fol1} says that the measure $\nu=g\cdot \mu_{\R^n}$ is a probability measure. So the equality \eqref{Change} is also true for $t=0$.

We introduce a unital normal $*$-monomorphism $J \co \L^\infty(\R^n) \to \L^\infty(\R^n) \otvn \L^\infty(\R^n,\nu)$, $f \mapsto f \ot 1$ which is trace preserving, the associated trace preserving normal faithful conditional expectation $\E \co \L^\infty(\R^n) \otvn \L^\infty(\R^n,\nu) \to \L^\infty(\R^n)$, $f \ot g \mapsto (\int_{\R^n} g \d \nu) f$ and finally for any $t \in \R$, the operator $U_t \co \L^\infty(\R^n) \otvn \L^\infty(\R^n,\nu) \to \L^\infty(\R^n) \ovl{\ot} \L^\infty(\R^n,\nu)$, $g \mapsto ((x,y) \mapsto g(x-ty,y))$. It is left to the reader to check that each $U_t$ is a $*$-automorphism of $\L^\infty(\R^n) \otvn \L^\infty(\R^n,\nu)$. For any positive function $g \in \L^1(\R^n \times \R^n,\mu_{\R^n} \ot \nu) \cap \L^\infty(\R^n \times \R^n,\mu_{\R^n} \ot \nu)$ and any $t \in \R$, we have using Fubini's Theorem twice and a change of variables in the second equality
\begin{align*}
\MoveEqLeft
  \int_{\R^n \times \R^n} \big[U_t(g)\big](x,y) \d \mu_{\R^n}(x) \d \nu(y)    
		=\int_{\R^n} \bigg(\int_{\R^n}  g(x-ty,y)\d \mu_{\R^n}(x)\bigg) \d \nu(y)\\
		&=\int_{\R^n \times \R^n} g(x,y) \d \mu_{\R^n}(x) \d \nu(y).
\end{align*} 
We deduce that $U_t$ is trace preserving. Moreover, for any $t,t' \geq 0$ and any $g \in \L^\infty(\R^n)\ovl{\ot} \L^\infty(\R^n,\nu)$, we have for almost all $(x,y) \in \R^n \times \R^n$
$$
\big[U_tU_{t'}(g)\big](x,y)        
=\big[U_{t'}(g)\big](x-ty,y)
=g\big(x-(t-t')y,y\big)
=\big[U_{t+t'}(g)\big](x,y).
$$
We deduce that $U_tU_{t'}=U_{t+t'}$. It is not difficult to check that the group $(U_t)_{t \in \R}$ is weak* continuous. Now, for any $t \geq 0$ and any $f \in \L^\infty(\R^n)$, we have for almost all $x \in \R^n$
\begin{align*}
\MoveEqLeft
\big[ \E U_tJ(f) \big](x)         
		=\int_{\R^n} \big[U_tJ(f)\big](x,y) \d \nu(y)
		=\int_{\R^n} \big[U_t(f \ot 1)\big](x,y) \d \nu(y)\\
		&=\int_{\R^n} (f \ot 1)(x-ty,y) \d \nu(y)
		=\int_{\R^n} f(x-ty) \d \nu(y)\\
		&=\int_{\R^n} f(x-ty)\frac{2}{\omega_{n+1}(1+|y|^2)^{\frac{n+1}{2}}} \d \mu_{\R^n}(y)
		=\big[\P_{\R^n,t}(f)\big](x).
\end{align*}
So we obtain a dilation of the semigroup $(\P_{\R^n,t})_{t \geq 0}$ for a suitable variant of Definition \ref{Def QWEP dilatable} for semifinite von Neumann algebras. 


\paragraph{Poisson semigroup on $\T^n$} 
Here we identify $\L^\infty(\T^n)$ with the space consisting of 1-periodic functions in the $n$ canonical directions on $\R^n$. If $t>0$, we introduce the function $\p_{\T^n,t} \co \T^n \to \C$ by 
$$
\p_{\T^n,t}(y)
=\sum_{m \in \Z^n} \e^{-2\pi t|m|} e_m(y), \quad y \in \R^n
$$
where $e_m(y)=\e^{2\i \pi m_1y_1} \cdots \e^{2\i \pi m_ny_n}$ and by $(\P_{\T^n,t})_{t \geq 0}$ the Poisson semigroup on $\L^\infty(\T^n)$. Recall that for any $t > 0$ and any $m \in \Z^n$ we have for almost all $x \in \R^n$
$$
\big[\P_{\T^n,t}(f)\big](x)
=(\p_{\T^n,t}*f)(x)
=\int_{[0,1]^n} \p_{\T^n,t}(s)f(x-y) \d y 
=\sum_{m \in \Z^n} \hat{f}(m)\e^{-2\pi t|m|} e_m(x).
$$
%
%
Recall that it is well-known that the Poisson summation formula,  \cite[(3.2.4)]{Gra1}, gives the following relation for any $t>0$
$$
\p_{\T^n,t}(y)
=\sum_{m \in \Z^n} \p_{\R^n,t}(y+m), \quad y \in \R^n.
$$
Consequently, for any $t > 0$ and any $f \in \L^\infty(\T^n)$, using Weil's formula in the third equality  and \eqref{Change} in the last equality, we have for any $t>0$ and for almost all $x \in \R^n$
\begin{align}
\label{computation-Poisson}
\MoveEqLeft
  \big[\P_{\T^n,t}(f)\big](x)          
		=\int_{[0,1]^n} p_{\T^n,t}(y) f(x-y) \d y
		= \int_{[0,1]^n} \sum_{m \in \Z^n} \p_{\R^n,t}(y+m)f(x-y) \d y\\
		&=\int_{\R^n} \p_{\R^n,t}(y)f(x-y) \d y 
		=\big[\P_{\R^n,t}\big(f\big)\big](x) 
		=\int_{\R^n} f(x-ty) \frac{2}{\omega_{n+1}(1+|y|^2)^{\frac{n+1}{2}}}\d \mu_{\R^n}(y). \nonumber
\end{align}
Now, it is easy to construct a dilation for $(\P_{\T^n,t})_{t \geq 0}$. Indeed, consider the probability measure $\nu=g\cdot \mu_{\R^n}$ introduced in the paragraph concerning the Poisson semigroup on $\R^n$. First note that the equality \eqref{computation-Poisson} is also true for $t=0$. Now, consider the trace preserving unital normal $*$-monomorphism $J \co  \L^\infty(\T^n) \to \L^\infty(\T^n) \otvn \L^\infty(\R^n,\nu)$, $f \mapsto f \ot 1$, the associated trace preserving normal faithful conditional expectation $\E \co \L^\infty(\T^n) \otvn \L^\infty(\R^n,\nu) \to \L^\infty(\T^n)$, $f \ot g \mapsto (\int_{\R^n} g \d \nu) f$ and finally for any $t \in \R$, the operator $U_t \co \L^\infty(\T^n) \otvn \L^\infty(\R^n,\nu) \to  \L^\infty(\T^n) \ovl{\ot} \L^\infty(\R^n,\nu)$, $g \mapsto ((x,y) \mapsto g(x-ty,y))$. It is left to the reader to check that these operators define a dilation of $(\P_{\T^n,t})_{t \geq 0}$.

\begin{rk}
In these particular cases, note that the induced dilations on the associated $\L^p$-spaces of both examples gives beautiful alternatives to the dilations provided by Fendler's theorem \cite{Fen1}. The question of finding dilations without ultraproduct techniques is implicitly raised in \cite[page 475]{HvNVW2}.
\end{rk}

\paragraph{Noncommutative Poisson semigroup on the von Neumann algebra $\VN(\F_n)$ of the free group $\F_n$} Let $n \geq 1$ be an integer. We denote by $\F_n$ a free group with $n$ generators denoted by $g_1,\ldots,g_n$. Any $s \in \F_n$ has a unique decomposition of the form
$$
s=g_{i_1}^{k_1}g_{i_2}^{k_2}\cdots g_{i_l}^{k_l},
$$
where $l \geq 0$ is an integer, each $i_j$ belongs to $\{1,\ldots,n\}$, each $k_j$ is a non-zero integer, and $i_j \not=i_{j+1}$ if $1\leq j \leq l-1$. The case when $l=0$ corresponds to the unit element $s=e_{\F_n}$. By definition, the length of $s$ is defined as $|s|=|k_1|+\cdots+|k_l|$. This is the number of factors in the above decomposition of $s$. For any nonnegative real number $t \geq 0$, we have a normal unital completely positive map $
 \P_{\F_n,t} \co \VN(\F_n) \to  \VN(\F_n)$, $\lambda_s \to \e^{-t|s|}\lambda_s$
These maps define a weak* continuous semigroup $(\P_{\F_n,t})_{t \geq 0}$ called the noncommutative Poisson semigroup, see \cite[Chapter 10]{JMX} and \cite{Haa1} for more information. In \cite[pages 107-108]{JMX}, it is proved that $\P_{\F_n,t}$ identifies with the free product $\ast_{1 \leq k \leq n}\P_{\T,\frac{t}{2 \pi}}$. Using free products of maps and the dilation of $(\P_{\T,t})_{t \geq 0}$ it is easy to construct a dilation for $(\P_{\F^n,t})_{t \geq 0}$.


\paragraph{Semigroups of hamiltonians} In this subsection, $M$ is a von Neumann algebra equipped with a normal faithful semifinite trace $\tau$. Let $L$ be a selfadjoint operator affiliated with $M$. Consider the Hamiltonian semigroup $(T_t)_{t \geq 0}$ of operators defined on $M$ by $T_t=\e^{-t(\ad L)^2}$ where
$$
(\ad L)(x)
=Lx-xL,\quad x \in \Dom(\ad L).
$$
It is known that $(T_t)_{t \geq 0}$ is a weak* continuous semigroup of completely positive maps, see \cite[Section 8.4]{JuX}, \cite[Section 8.B]{JMX} and \cite[Example 30.1]{Par1}. Moreover, if $g$ is a Gaussian variable on some probability space $(\Omega,\mu)$ with mean zero and variance $\sqrt 2$, by \cite[Section (8.2)]{JuX} we have the following formula for any $t \geq 0$
$$
T_t(x)
= \E \big[\e^{\i \sqrt{t}  gL} x\, \e^{-\i \sqrt{t}gL}\big],\quad x \in M
$$
where $\i^2=-1$ and where $\E$ denotes the expectation with respect to $\Omega$.

We introduce the von Neumann algebra $N=\L^\infty(\Omega)\otvn M$ equipped with the faithful semifinite normal trace $\tau_N=(\int_{\Omega}\cdot\ \d\mu) \ot \tau$. Note that, by \cite[page 41]{BLM}, we have a $*$-isomorphism $N=\L^\infty(\Omega,M)$. We consider the canonical trace preserving normal unital $*$-monomorphism $J \co M \to \L^\infty(\Omega) \otvn M$, $x \mapsto 1 \ot x$ and the associated trace preserving normal faithful conditional expectation $\E \co \L^\infty(\Omega) \otvn M \to M$, $f \ot x \mapsto (\int_{\Omega} \cdot \d \mu) x $. For any $t \geq 0$, we define the element $D_t$ of $N=\L^\infty(\Omega,M)$ by
$$
D_t(\omega)
=\e^{\i \sqrt{t}g(\omega)L}.
$$
Note that each $D_t$ is a unitary element of $N=\L^\infty(\Omega,M)$. Now, for any $t \geq 0$ we define the $*$-automorphism of $U_t \co \L^\infty(\Omega,M) \to \L^\infty(\Omega,M)$, $f \mapsto D_t fD_t^*$ of $N$. For any positive element $f \in \L^1(N) \cap N $ and any $t \in \R$, we have 
\begin{align*}
\MoveEqLeft
 \int_{\Omega} \tau  \big(\big[U_t(f)\big](\omega)\big) \d\mu(\omega)      
		=\int_{\Omega} \tau\big(\big[D_t fD_t^*\big](\omega)\big) \d\mu(\omega)\\
		&=\int_{\Omega} \tau\big(D_t(\omega) f(\omega)D_t^*(\omega)\big) \d\mu(\omega)
		=\int_{\Omega} \tau\big(f(\omega)\big) \d\mu(\omega).
\end{align*} 
We deduce that each map $U_t$ is trace preserving. It is not difficult to check that the group $(U_t)_{t \in \R}$ is weak* continuous. Finally, for any $x \in M$ and any $t \geq 0$, we have
\begin{align*}
\MoveEqLeft
 \E U_tJ(x)  
   =\E U_t(1\ot x)
	=\int_{\Omega} U_t(1\ot x)\d\mu(\omega)  
	=\int_{\Omega}  D_t(\omega) (1 \ot x)D_t(\omega)^*\d\mu(\omega) \\
   &= \int_{\Omega} \e^{\i \sqrt{t}g(\omega)L}(1 \ot x) \e^{-\i \sqrt{t}g(\omega)L}\d\mu(\omega)
	=\E \big[\e^{\i \sqrt{t} gL} x\, \e^{-\i \sqrt{t}gL}\big]
	=T_t(x).
\end{align*}
For any $t\geq 0$, we conclude that
$$
T_t
=\E U_tJ.
$$

%

%
%
\section{Dilations of semigroups on noncommutative $\L^p$-spaces}
\label{sec:dilations-Lp}

The goal is to prove Theorem \ref{thm-dilation-Fendler} below which is a noncommutative $\L^p$ variant of Theorem \ref{thm-dilation-semigroup-weak-star2}. Suppose $1 \leq p<\infty$. Recall the definition of \cite[page 239]{JLM} which says that a contraction $T \co \L^p(M) \to \L^p(M)$ on a noncommutative $\L^p$-space $\L^p(M)$ is dilatable if there exist a noncommutative $\L^p$-space $\L^p(N)$, two contractions $J \co \L^p(M) \to \L^p(N)$ and $P \co \L^p(N) \to \L^p(M)$ and an isometry $U \co \L^p(N) \to \L^p(N)$ such that $T^k=PU^kJ$ for any integer $k \geq 0$. Note that Akcoglu's theorem implies that any positive contraction on a commutative $\L^p$-space $\L^p(\Omega)$ is dilatable. Now, we introduce a variant. 

\begin{defi}
Suppose $1\leq p<\infty$. We say that a completely positive contraction $T \co \L^p(M) \to \L^p(M)$ on a noncommutative $\L^p$-space $\L^p(M)$ is completely positively dilatable if there exist a noncommutative $\L^p$-space $\L^p(N)$, two completely positive contractions  $J \co \L^p(M) \to \L^p(N)$ and $P \co \L^p(N) \to \L^p(M)$ and a completely positive invertible isometry $U \co \L^p(N) \to \L^p(N)$ with $U^{-1}$ completely positive such that 
$$
T^k=PU^kJ,\qquad k \geq 0.
$$
\end{defi}

\begin{rk}
Note that a dilatable $\phi$-Markov $T \co M \to M$ on a von Neumann algebra $M$ equipped with a normal faithful state $\phi$ induces a completely positively dilatable contraction $T_p \co \L^p(M) \co \to \L^p(M)$ on the associated noncommutative $\L^p$-space $\L^p(M)$. 
\end{rk}

In this section, we use Banach ultraproducts. The same method that the  beginning of the proof of Theorem \ref{thm-dilation-semigroup-weak-star2} together with the stability of the class of noncommutative $\L^p$-spaces under Banach ultraproducts \cite{Ray1} gives the following result.
\begin{lemma}
\label{lemma-dilation-on-Q}
Suppose $1 \leq p <\infty$. Consider a (not necessarily strongly continuous) semigroup $(T_t)_{t \geq 0}$ of completely positive contractions on $\L^p(M)$ such that each operator $T_t$ is completely positively dilatable. Then there exists a noncommutative $\L^p$-space $\L^p(\widetilde{N})$, a group $(\widetilde{S}_t)_{t \in \Q}$ of completely positive invertible isometries of $\L^p(\widetilde{N})$, two completely positive contractions $\widetilde{J} \co \L^p(M) \to \L^p(\widetilde{N})$ and $\widetilde{P} \co \L^p(\widetilde{N}) \to \L^p(M)$ such that 
$$
T_{t} 
=\widetilde{P} \widetilde{S}_{t} \widetilde{J}, \qquad  t \in \mathbb{Q}^+.
$$
\end{lemma}



One more time, if the semigroup $(T_t)_{t\geq 0}$ is strongly continuous, the above ultraproduct construction yields a too large space $\L^p(\widetilde{N})$. So we cannot expect that the representation $\widetilde{U} \co t \mapsto \widetilde{U}_{t}$ of $\mathbb{Q}$ to be continuous on $\L^p(\widetilde{N})$. However, it is still possible to restrict $t \mapsto \widetilde{U}_{t}$ to a smaller subspace on which the desired continuity holds. We skip the end of the proof of the following result.


\begin{thm}
\label{thm-dilation-Fendler}
Suppose $1< p < \infty$. Let $(T_t)_{t \geq 0}$ be a strongly continuous semigroup of completely positive contractions on a noncommutative $\L^p$-space $\L^p(M)$ such that each $T_t \co \L^p(M) \to \L^p(M)$ is completely positively dilatable. Then there exists a noncommutative $\L^p$-space $\L^p(N)$, a strongly continuous group of completely positive isometries $U_t \co \L^p(N) \to \L^p(N)$ and two completely positive contractions $J \co \L^p(M) \to \L^p(N)$ and $P \co \L^p(N) \to \L^p(M)$ such that
\begin{equation*}
	\label{equa-dilatation-semigroupe}
	T_t=PU_tJ,\qquad t \geq 0.
\end{equation*}
\end{thm}

\begin{proof}
By Lemma \ref{lemma-dilation-on-Q}, we obtain a representation $\widetilde{S} \co \mathbb{Q} \to \B(\L^p(\widetilde{N}))$ by completely positive isometric operators. We have
$$
T_{t} 
=\widetilde{P} \widetilde{S}_{t} \widetilde{J}, \qquad  t \in \mathbb{Q}^+.
$$
By Theorem \ref{DLG}, we deduce that the kernel $\mathcal{K}(\widetilde{S})$ of $\mathcal{S}^c(\widetilde{S})$ contains a unique element $Q_c$ and that $Q_c \co \L^p(\widetilde{N}) \to \L^p(\widetilde{N})$ is a projection from the Banach space $\L^p(\widetilde{N})$ onto the subspace $\L^p(\widetilde{N})_c$ of continuously translated elements. Furthermore, we have
$$
\widetilde{S}_tQ_c
=Q_c\widetilde{S}_t,\qquad t \in \mathbb{Q}.
$$
By \cite[Theorem 2.22]{AbA}, the range $\L^p(\widetilde{N})_c$ of the projection $Q_c$ is invariant under the operator $\widetilde{S}_t$ for any $t \in \Q$. Moreover, we infer from Lemma \ref{lemma-continuité-uniform-convexity} that the range $\Ran (\widetilde{J})$ of the map $\widetilde{J}$ given by Lemma \ref{lemma-dilation-on-Q} is contained in the subspace $\L^p(\widetilde{N})_c$ of continuously translated elements of $\L^p(\widetilde{N})$ of the representation $\widetilde{S}$. Furthermore, since each operator $\widetilde{S}_t$ is isometric and completely positive, hence contractive, we see that the convex semigroup $\mathcal{S}^c(\widetilde{S})$ of $\widetilde{S}$ over the identity consists of completely positive contractions only by using \cite[Lemma 2.6]{ArK} and the weak lower semicontinuity of the norm. It follows that $Q_c$ is also contractive and completely positive and consequently that the subspace $\L^p(\widetilde{N})_c$ is 1-completely positively complemented in $\L^p(\widetilde{N})$, hence completely isometric and completely order isomorphic to a noncommutative $\L^p$-space $\L^p(N)$ by the main result of \cite{ArR}. Now, for any $t \in \mathbb{Q}$ we define the completely positive and isometric map
\begin{equation*}
\label{def-de-U}
	U_t
	=\widetilde{S}_{t}|\L^p(\widetilde{N})_c \co \L^p(\widetilde{N})_c \to\L^p(\widetilde{N})_c.
\end{equation*}
Note that the canonical map $J \co \L^p(M) \to \L^p(\widetilde{N})_c=\L^p(N)$ of $\L^p(M)$ is contractive and completely positive. Finally, we consider the restriction $P=\widetilde{P}|\L^p(N) \co \L^p(N) \to \L^p(M)$ of $\widetilde{P}$ on $\L^p(N)$. 
So we have the following diagram. 
$$
 \xymatrix @R=1cm @C=2cm{
\L^p(\widetilde{N})\ar[r]^{\widetilde{S}_{t}}\ar@/^2pc/[r]^{Q_c}&\L^p(\widetilde{N})\ar@/^4pc/[dd]^{\widetilde{P}}\\
    \L^p(\widetilde{N})_c=\L^p(N)   \ar@{^{(}->}[u]   \ar[r]^{U_t=\widetilde{S}_{t}|\L^p(\widetilde{N})_c}  &\L^p(N)=\L^p(\widetilde{N})_c \ar[d]^{P}\ar@{^{(}->}[u] \\
    \L^p(M)    \ar[r]_{T_{t}}\ar[u]^{J} \ar@/^4pc/[uu]^{\widetilde{J}}&  \L^p(M)
}
$$
Now, it is left to the reader to finish the proof. 
%
\end{proof}


\begin{rk} Suppose $1<p<\infty$ with $p\not=2$. Note that there exists a completely positive contractive map $T \co \S^p \to \S^p$ which does not admit an isometric dilation on a noncommutative $\L^p$-space, see \cite{JLM}. It would be interesting (if it does occur) to find a completely positive contractive Schur multiplier $M_A \co \S^p \to \S^p$ without isometric dilation on a noncommutative $\L^p$-space or a completely positive contractive Fourier multiplier $M_t \co \L^p(\VN(G)) \to \L^p(\VN(G))$ without isometric dilation (on a necessarily nonabelian group $G$ due to Akcoglu's theorem). See also \cite{Arh1}, \cite{ALM} and \cite{AFM} for more information on dilations on noncommutative $\L^p$-spaces.
\end{rk}

\section{Semigroups of selfadjoint Fourier multipliers}
\label{sec:exemples}

As we said in the introduction, Haagerup and Musat \cite[Theorem 4.4]{HaM} have characterised dilatable Markov maps. Indeed, they proved that if $T \co M \to M$ is a $\phi$-Markov map on a von Neumann algebra $M$ equipped with a state $\phi$ then $T$ is dilatable if and only if $T$ is factorizable in the sense of \cite{AnD}. This result allows us to give concrete examples of dilatable semigroups.

Suppose that $G$ is a discrete group. Recall that we denote by $e$ the neutral element of $G$. We denote by $\lambda_s \co \ell^2_G \to \ell^2_G$ the unitary operator of left translation by $s$ and $\VN(G)$ the von Neumann algebra of
$G$ spanned by the $\lambda_s$'s where $s \in G$. It is an finite von Neumann algebra with its canonical faithful normal finite trace given by
$$
\tau_{G}(x)
=\big\langle\epsi_{e},x(\epsi_{e})\big\rangle_{\ell^2_G}
$$
where $(\epsi_s)_{s \in G}$ is the canonical basis of $\ell^2_G$ and $x \in \VN(G)$. A Fourier multiplier is a normal linear map $T \co \VN(G) \to \VN(G)$ such that there exists a complex function $t \co G \to \C$ such that $T(\lambda_s)=t_s\lambda_s$ for any $s \in G$. In this case, we denote $T$ by $M_t \co  \VN(G) \to \VN(G)$. It is well-known that a Fourier multiplier $M_t \co \VN(G) \to \VN(G)$ is completely positive if and only if the function $t$ is positive definite. It is easy to see that a $\tau_G$-Markov Fourier multiplier $M_t \co \VN(G) \to \VN(G)$ is selfadjoint if and only if $t \co G\to \C$ is a real function.

Using the factorisability of selfadjoint $\tau_G$-Markov Fourier multipliers of \cite{Ric} (see \cite{ArK} for a generalization), we deduce the following result:
\begin{cor}
\label{thm-Dilation-semigroup-Fourier-multipliers}
Let $G$ be a discrete group. Let $(T_t)_{t \geq 0}$ be a weak* continuous semigroup of selfadjoint $\tau_G$-Markov Fourier multipliers on the von Neumann algebra $\VN(G)$. Then the semigroup $(T_t)_{t \geq 0}$ is dilatable. 
\end{cor}

We refer to the preprint \cite{Arh4} for another approach which can be used for the class of unimodular locally compact groups.

%

%
%
%

\section{Applications to $\H^\infty$ functional calculus}
\label{sec:Applications}

We start with a little background on sectoriality and $\H^\infty$ functional calculus. We refer to \cite{Haa}, \cite{KW}, \cite{JMX}, \cite{HvNVW2} and \cite{Arh2} for details and complements. Let $X$ be a Banach space. A closed densely defined linear operator $A \co \Dom(A)\subset X \to X$ is called sectorial of type $\omega$ if its spectrum $\sigma(A)$ is included in the closed sector $\overline{\Sigma_\omega}$, and for any angle $\omega<\theta < \pi$, there is a positive constant $K_\theta$ such that
\begin{equation*}\label{Sector}
 \bnorm{(\lambda-A)^{-1}}_{X\to X}\leq
\frac{K_\theta}{\vert  \lambda \vert},\qquad \lambda
\in\mathbb{C}-\overline{\Sigma_\theta}.
\end{equation*}
If $-A$ is the negative generator of a strongly continuous bounded semigroup on a $X$ then $A$ is sectorial of type $\frac{\pi}{2}$. By \cite[Example 10.1.3]{HvNVW2}, sectorial operators of type $<\frac{\pi}{2}$ coincide with negative generators of bounded analytic semigroups.

For any $0<\theta<\pi$, let $\H^{\infty}(\Sigma_\theta)$ be the algebra of all bounded analytic functions $f \co  \Sigma_\theta\to \C$, equipped with the supremum norm $\norm{f}_{\H^{\infty}(\Sigma_\theta)}=\,\sup\bigl\{\vert f(z)\vert \, :\, z\in \Sigma_\theta\bigr\}$. Let $\H^{\infty}_{0}(\Sigma_\theta)\subset \H^{\infty}(\Sigma_\theta)$ be the subalgebra of bounded analytic functions $f \co \Sigma_\theta\to \C$ for which there exist $s,c>0$ such that $\vert f(z)\vert\leq \frac{c \vert z \vert^s}{(1+\vert z \vert)^{2s}}$ for any $z \in \Sigma_\theta$. Finally, we let $\H^\infty_0(\Sigma_{\theta+})=\cup_{\omega>\theta} \H^\infty_0(\Sigma_{\theta})$. 

Given a sectorial operator $A$ of type $0< \omega < \pi$, a bigger angle $\omega<\theta<\pi$, and a function $f\in \H^{\infty}_{0}(\Sigma_\theta)$, one may define a bounded operator $f(A)$ by means of a Cauchy integral (see e.g. \cite[Section 2.3]{Haa} or \cite[Section 9]{KW}). The resulting mapping $\H^{\infty}_{0}(\Sigma_\theta) \to \B(X)$ taking $f$ to $f(A)$ is an algebra homomorphism. By definition, $A$ has a bounded $\H^{\infty}(\Sigma_\theta)$ functional calculus provided that this homomorphism is bounded, that is if there exists a positive constant $C$ such that $\bnorm{f(A)}_{X \to X} \leq C\norm{f}_{\H^{\infty}(\Sigma_\theta)}$ for any $f \in \H^{\infty}_{0}(\Sigma_\theta)$. In the case where $A$ has a dense range, the latter boundedness condition allows a natural extension of $f\mapsto f(A)$ to the full algebra $\H^{\infty}(\Sigma_\theta)$.

Suppose $1 \leq p<\infty$. In the sequel, we say that a sectorial operator $A$ on a vector-valued noncommutative $\L^p$-space $\L^p(M,E)$ admits a completely bounded $\H^{\infty}(\Sigma_\theta)$ functional calculus if $\Id_{\mathrm{S}^p} \ot A$ admits a bounded $\H^{\infty}(\Sigma_\theta)$ functional calculus on the vector-valued Schatten space $S^p(\L^p(M,E))$.

Using the connection between the existence of dilations in UMD spaces and $\H^\infty$ functional calculus together with the well-known angle reduction principle of Kalton-Weis relying on R-sectoriality, Theorem \ref{thm-dilation-Fendler} allows us to recover the last page of the memoir \cite[page 125]{JMX}:

\begin{thm}
\label{Th-funct-calculus}
Let $M$ be a von Neumann algebra equipped with a normal faithful state $\phi$. Let $(T_t)_{t \geq 0}$ be a weak* continuous semigroup of selfadjoint factorizable $\phi$-Markov maps on $M$. Suppose $1<p<\infty$. We let $-A_p$ be the generator of the induced strongly continuous semigroup $(T_{t,p})_{t\geq 0}$ on the Banach space $\L^p(M)$. Then for any $\theta>\pi|\frac{1}{p}-\frac{1}{2}|$, the operator $A_p$ has a completely bounded $\H^{\infty}(\Sigma_\theta)$ functional calculus.
\end{thm}

\begin{proof}
Using Theorem \ref{thm-dilation-semigroup-weak-star}, with obvious notations, note that we have a dilation of the strongly continuous semigroup $(\Id_{\S^p} \ot T_{t,p})_{t \geq 0}$ acting on the Banach space $\S^p(\L^p(M))$: 
$$
\Id_{\S^p} \ot T_{t,p}
=(\Id_{\S^p} \ot \E_{p})(\Id_{\S^p} \ot U_{t,p})(\Id_{\S^p} \ot J_p)
$$
by a strongly continuous group $(\Id_{\S^p} \ot U_{t,p})_{t \in \R}$ of isometries acting on $\S^p(\L^p(N))$. Recall that a noncommutative $\L^p$-space is a UMD Banach space for any $1 < p< \infty$ by \cite[Corollary 7.7]{PiX}. Hence $\S^p(\L^p(N))$ is UMD. Now, transference \cite[Corollary 10.9]{KW} gives the existence of some $0<\theta<\pi$ such that $\Id_{S^p} \ot A_p$ admits a bounded $\H^{\infty}(\Sigma_\theta)$ functional calculus. Now, we reduce the angle and conclude with \cite[Proposition 5.8]{JMX} since each $T_t$ is selfadjoint.
\end{proof}

This theorem is applicable to any weak* continuous semigroup $(T_t)_{t \geq 0}$ of selfadjoint $\tau_G$-Markov Fourier multipliers on the von Neumann algebra $\VN(G)$ of a discrete group $G$.

Theorem \ref{Th-funct-calculus} combined with the proof of \cite[Theorem 10.4.16]{HvNVW2} and some results of \cite{JMX} describing square functions in the noncommutative setting, imply Theorem \ref{Th-equivalence-square functions} below in the spirit of \cite[Corollary 7.7]{JMX} and \cite[Corollary 7.10]{JMX}. We skip the details but we explain some notations. Let $(\Omega,\mu)$ be a $\sigma$-finite measure space. By a subpartition of $\Omega$, we mean a finite set $\pi=\{I_1,\ldots, I_m\}$ of pairwise disjoint measurable subsets of $\Omega$ such that $0<\mu(I_i)<\infty$ for any $1\leq i\leq m$. Let $X$ be a Banach space and let $\pi$ be a subpartition of $\Omega$. For any $f \in \L^p(\Omega,X)$ we let
\begin{equation*}
f_\pi
=\sum_{i=1}^{m} \frac{1}{\mu(I_i)}\biggl(\int_{I_i}
f \biggr)\, 1_{I_i}.
\end{equation*}
Here $1_I$ denotes the indicator function if $I$. Finally, we can consider some limits $\lim_\pi$ if subpartitions are directed by refinement. We refer to \cite{JMX} for more details and for the links with noncommutative square functions.

\begin{thm}
\label{Th-equivalence-square functions}
Let $M$ be a von Neumann algebra equipped with a normal faithful state $\phi$. Let $(T_t)_{t \geq 0}$ be a weak* continuous semigroup of selfadjoint factorizable $\phi$-Markov maps on $M$. We let $-A_p$ be the generator of the strongly continuous semigroup $(T_{t,p})_{t\geq 0}$ induced on the noncommutative $\L^p$-space $\L^p(M)$. Let $F$ be a function of $\H^\infty_0(\Sigma_{\theta+})-\{0\}$.
\begin{enumerate}
	\item Suppose $1 < p \leq 2$. Then for any $x \in \Dom(A_p) \cap \Ran(A_p)$, we have an equivalence
	\begin{align*}
  \norm{x}_{\L^p(M)}
	&\approx \inf \Bigg\{ \lim_\pi \Bgnorm{\left(\int_{0}^{\infty} \big(F(tA_p)x_1\big)_\pi^*\big(F(tA_p)x_1\big)_\pi\frac{\d t}{t}\right)^{\frac{1}{2}}}_{\L^p(M)}+ \\
		& \lim_\pi \Bgnorm{\left(\int_{0}^{\infty} \big(F(tA_p)x_1\big)_\pi\big(F(tA_p)x_1\big)_\pi^*\frac{\d t}{t}\right)^{\frac{1}{2}}}_{\L^p(M)}\Bigg\},
\end{align*} 
where the infimum runs over all $x_1,x_2 \in \L^p(M)$ such that $x=x_1+x_2$.	
	\item Suppose $2 \leq p <\infty$. Then for any $x \in \Dom(A_p) \cap \Ran(A_p)$, we have an equivalence
	\begin{align*}
\norm{x}_{\L^p(M)}
&\approx \max\Bigg\{ \lim_{\substack{\alpha\to 0\\ \beta\to +\infty}} \Bgnorm{\left(\int_{\alpha}^{\beta} \big(F(tA_p)x\big)^*\big(F(tA_p)x\big)\frac{\d t}{t}\right)^{\frac{1}{2}}}_{\L^p(M)},\\
		&\lim_{\substack{\alpha\to 0\\ \beta\to +\infty}}\Bgnorm{\left(\int_{\alpha}^{\beta} \big(F(tA_p)x\big)\big(F(tA_p)x\big)^*\frac{\d t}{t}\right)^{\frac{1}{2}}}_{\L^p(M)}\Bigg\}.
\end{align*} 
	
\end{enumerate}
\end{thm}

Suppose $1 \leq p \leq \infty$. Recall that the vector-valued noncommutative space $\L^p(M,E)$ is well-defined \cite{Pis5} if $M$ is a hyperfinite semifinite von Neumann algebra and if $E$ is an operator space. The ideas of the manuscript \cite{Jun2} (which unfortunately seems definitely postponed) allows to define $\L^p(M,E)$ beyond the hyperfinite case for a QWEP von Neumann algebra $M$ and a locally-$\mathrm{C}^*(\F_{\infty})$ operator space $E$. Using Remark \ref{Remarque-QWEP}, we can give vector-valued variants of Theorem \ref{Th-funct-calculus} using the results of \cite{Arh2} (see also \cite{Arh3} for related things) and classical principles. 

For that, we need of the operator space analog $\OUMD_p$ \cite[Definition 4.8]{Pis5} of the Banach space classical property UMD \cite[Definition 4.2.1]{HvNVW2} and a more contraignant variant introduced in \cite{Arh2}. Suppose $1< p< \infty$. Let $E$ be a locally-$\mathrm{C}^*(\F_{\infty})$ operator space. We say that $E$ is $\OUMD_p'$ if there exists a positive constant $K$ such that for any QWEP von Neumann algebra $M$ equipped with a normal faithful state, any positive integer $n$, any finite martingale $(x_k)_{0 \leq k \leq n}$ in $\L^p(M,E)$ relative to a filtration $(M_k)_{0 \leq k \leq n}$ and any choice of signs $\epsi_1,\ldots,\epsi_{n}\in \{\pm 1\}$ we have
$$
\Bgnorm{\sum_{k=1}^n \epsi_k \d x_k}_{\L^p(M, E)}
\leq K \Bgnorm{\sum_{k=1}^n \d x_k}_{\L^p(M,E)}.
$$
The property $\OUMD_p$ is defined similarly but with hyperfinite and finite von Neumann algebras. Finally, for any index set $I$, we denote by $\mathrm{OH}(I)$ the associated operator Hilbert space introduced by Pisier. Finally, recall that the definition of $\QWEP$-factorizability is given in Remark \ref{Remarque-QWEP}.

Then we can obtain the following theorem by transference. 

\begin{thm}
\label{Th-funct-calculus-B}
Let $M$ be a $\QWEP$ von Neumann algebra equipped with a normal faithful state $\phi$. Let $(T_t)_{t \geq 0}$ be a weak* continuous semigroup of $\QWEP$-factorizable $\phi$-Markov maps on $M$. Let $E$ be an $\OUMD_p$ locally-$\mathrm{C}^*(\F_{\infty})$ operator space. Suppose $1<p<\infty$. We let $-A_p$ be the generator of the induced strongly continuous semigroup $(T_{t,p} \ot \Id_E)_{t \geq 0}$ on the Banach space $\L^p(M,E)$. Then the operator $A_p$ has a completely bounded $\H^{\infty}(\Sigma_\theta)$ functional calculus for some $0<\theta <\pi$.
\end{thm}

\begin{proof}
With obvious notations, observe that we have a dilation of the strongly continuous semigroup $(\Id_{\S^p} \ot T_{t,p} \ot \Id_E)_{t \geq 0}$ acting on $\S^p(\L^p(M,E))$: 
$$
\Id_{\S^p} \ot T_{t,p} \ot \Id_E
=(\Id_{\S^p} \ot \E_{p} \ot \Id_E)(\Id_{\S^p} \ot U_{t,p} \ot \Id_E)(\Id_{\S^p} \ot J_p \ot \Id_E)
$$
by a strongly continuous group $(\Id_{\S^p} \ot U_{t,p} \ot \Id_E)_{t \in \R}$ of isometries acting on $\S^p(\L^p(N,E)))$ where $N$ is a $\QWEP$ von Neumann algebra. By \cite[Proposition 3.12]{Mus1}, the operator space $\S_I^p(E)$ is $\OUMD_p$ for any index set $I$. So by \cite[Proposition 3.12]{Mus1} again, the operator space $(\S^p_I(E))^\ul$ is also $\OUMD_p$. Finally, by definition of $\L^p(N,E)$, we infer that the Banach space $\S^p(\L^p(N,E)))$ is UMD. Now, we conclude by transference \cite[Corollary 10.9]{KW}.
\end{proof}

Combining with the result \cite[Theorem 1.4]{Arh2}\footnote{\thefootnote. Note that the word ``selfadjoint'' is missing in the assumptions of \cite[Theorem 1.4, Theorem 1.6 and Theorem 3.6]{Arh2}.} we obtain the following result. We skip the details.

\begin{thm}
Let $M$ be a $\QWEP$ von Neumann algebra equipped with a normal faithful state. Let $(T_t)_{t\geq 0}$ be a weak* continuous semigroup of selfadjoint $\QWEP$-factorizable $\phi$-Markov on $M$. Suppose $1< p,q<\infty$ and $0<\alpha<1$. Let $E$ be an operator space such that $E=\big(\mathrm{OH}(I),F\big)_\alpha$ for some index set $I$ and for some $\OUMD_{q}'$ operator space $F$ with $\frac{1}{p}=\frac{1-\alpha}{2}+\frac{\alpha}{q}$. We let $-A_p$ be the generator of the strongly continuous semigroup $(T_{t,p} \ot \Id_{E})_{t \geq 0}$ on the vector-valued noncommutative $\L^p$-space $\L^p(M,E)$. Then for some $0<\theta<\frac{\pi}{2}$, the operator $A_p$ has a completely bounded $\H^{\infty}(\Sigma_\theta)$ functional calculus.
\end{thm}

\vspace{0.3cm}

\textbf{Acknowledgements}.
The author would like to thank Yves Raynaud, Magdalena Musat and Yoshimichi Ueda for some discussions and Kay Schwieger for comments on the paper. Finally, we would like to thank the referee for his comments.


\vspace{0.5cm}
\footnotesize{ \noindent 
13 rue Didier Daurat, 81000 Albi, France\\
cedric.arhancet@protonmail.com\\
URL: \href{http://sites.google.com/site/cedricarhancet}{https://sites.google.com/site/cedricarhancet}\hskip.3cm
}

\end{document}